\documentclass[hidelinks,onefignum,onetabnum]{siamart220329}

\usepackage{lipsum}
\usepackage{amsfonts}
\usepackage{graphicx}
\usepackage{epstopdf}
\usepackage{algorithmic}
\ifpdf
  \DeclareGraphicsExtensions{.eps,.pdf,.png,.jpg}
\else
  \DeclareGraphicsExtensions{.eps}
\fi

\newsiamremark{remark}{Remark}
\newsiamremark{hypothesis}{Hypothesis}
\crefname{hypothesis}{Hypothesis}{Hypotheses}
\newsiamthm{claim}{Claim}

\headers{Controlling Nutrition and Propulsion Force in Running Races}{C. Cook, S. Lenhart, W. Hager, and G. Chen}

\title{Optimally Controlling Nutrition and Propulsion Force in a Long Distance Running Race\thanks{Submitted to the editors 06/29/22.
\funding{Support by the US National
Science Foundation under grant 2031213 and by the US Office of
Naval Research under grant N00014-22-1-2397 are gratefully acknowledged.}}}

\author{Cameron Cook\thanks{RTI Health Solutions, Research Triangle Park, NC 
  (\email{ccook@rti.org}, \url{https://www.rtihs.org/research-team/cameron-cook}).}
\and Suzanne Lenhart \thanks{Department of Mathematics, University of Tennessee, Knoxville, TN 
  (\email{slenhart@utk.edu}).}%\footnotemark[3]
\and William Hager \thanks{Department of Mathematics, University of Florida, Gainesville, FL
(\email{hager@ufl.edu}).} %\footnotemark[4]
\and Guoxun Chen \thanks{Department of Nutrition, University of Tennessee, Knoxville, TN
(\email{gchen@utk.edu}).}} 
\usepackage{amsopn}

\ifpdf
\hypersetup{
  pdftitle={Optimally Controlling Nutrition and Propulsion Force in a Long Distance Running Race},
  pdfauthor={C. Cook, S. Lenahrt, W. Hager, and G. Chen}
}
\fi

\begin{document}

%\title{Optimally Controlling Nutrition and Propulsion Force in a Long Distance Running Race}
%
%\author[1]{Cameron Cook*}
%
%\author[2]{Suzanne Lenhart}
%
%\author[3]{Guoxun Chen}
%
%\author[4]{William Hager}
%
%\authormark{Cook \textsc{et al}}
%
%
%\address[1]{\orgdiv{Health Economics}, \orgname{Research Triangle Institute}, \orgaddress{\state{North Carolina}, \country{United States}}}
%
%\address[2]{\orgdiv{Mathematics Department}, \orgname{University of Tennessee}, \orgaddress{\state{Tennessee}, \country{United States}}}
%
%\address[3]{\orgdiv{Nutrition Department}, \orgname{University of Tennessee}, \orgaddress{\state{Tennessee}, \country{United States}}}
%
%\address[4]{\orgdiv{Mathematics Department}, \orgname{University of Florida}, \orgaddress{\state{Florida}, \country{United States}}}
%
%\corres{*Cameron Cook, Corresponding address. \email{ccook54@vols.utk.edu}}
%
%\presentaddress{3040 E Cornwallis Rd, Research Triangle, NC 27709}

\maketitle
% REQUIRED
\begin{abstract}
Runners competing in races are looking to optimize their performance. In this paper, a runner's performance in a race, such as a marathon, is formulated as an optimal control problem where the controls are: the nutrition intake throughout the race and the propulsion force of the runner.  As nutrition is an integral part of successfully running long distance races, it needs to be included in models of running strategies.  We formulate a system of ordinary differential equations to represent the velocity, fat energy, glycogen energy, and nutrition for a runner competing in a long-distance race. The energy compartments represent the energy sources available in the runner’s body.  We allocate the energy source from which the runner draws, based on how fast the runner is moving.  The food consumed during the race is a source term for the nutrition differential equation. With our model, we are investigating strategies to manage the nutrition and propulsion force in order to minimize the running time in a fixed distance race.  This requires the solution of a nontrivial singular control problem.  Our results confirm the belief that the most effective way to run a race is to run approximately the same pace the entire race without letting one's energies hit zero.
\end{abstract}

% REQUIRED

% REQUIRED

\section{Introduction}\label{Introduction}
Running is one of the most popular forms of exercise. There are more than $275,000$ road races per year in the United States \cite{RunUSA}.  At the marathon distance alone, there are over $500,000$ people a year in the United States who choose to race (www.runningtheusa.com).  Some people run for fun while others choose to seriously compete, attempting to run the shortest time for the fixed distance race. Outside of sprints, pacing oneself to run the best possible race is crucial. If the runner starts the race too slow, they may not finish in as fast a time as expected. On the other-hand, if the runner goes out too fast, they may find themselves running out of energy, struggling to even finish the race. Knowing what pace strategy is best for each individual runner is the hard part, as it depends on several values that are individual to each runner, something this work aims to answer. Looking at data collected from timing mats in the 2015 Boston Marathon, Figure \ref{Boston} plots how the runners  placed versus how their pace differed from their first 10k and their overall pace. Amongst this elite field, one can see that on average, the runners who placed lower had much higher pace variation than runners who finished in a shorter time (better placing).\\
\\
\begin{figure}[htbp]
%\caption{Energy-Nutrition System}
\begin{center}
    \includegraphics[scale=.6]{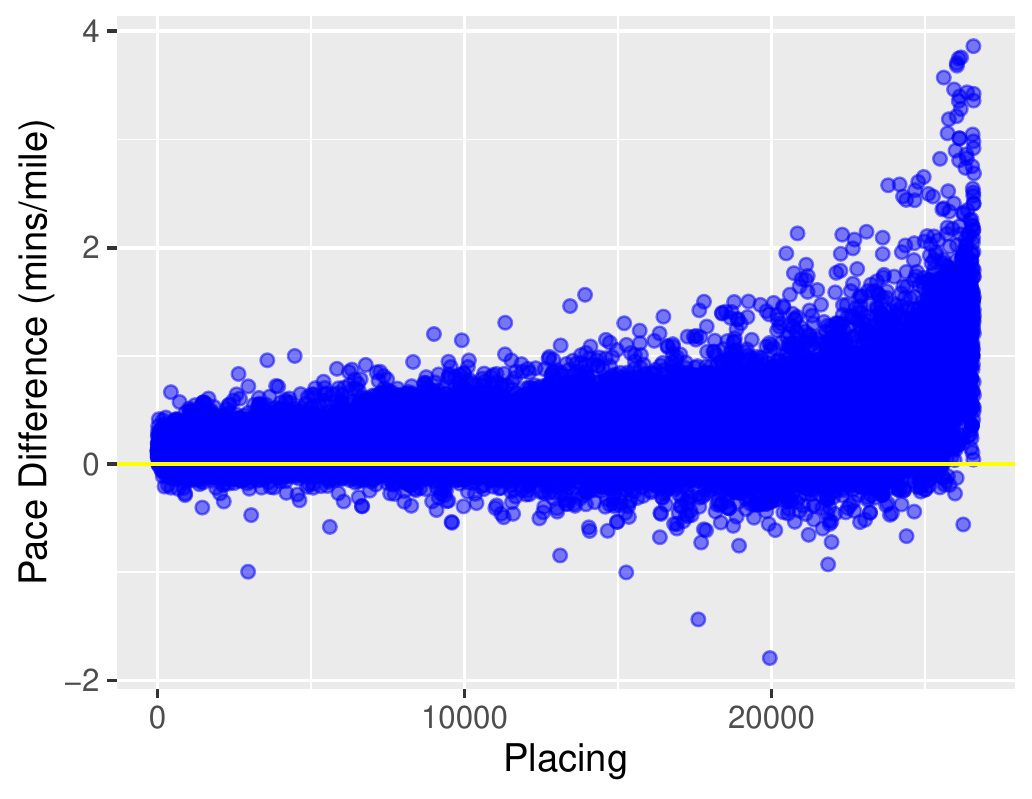} 

\end{center}
\caption{The above data comes from the 2015 Boston Marathon. We graphed finish placement on the x-axis and pace difference on the y-axis. Pace difference is the difference between the runners 5km pace (min/mile) and the runners overall pace (min/mile) for the race. The horizontal line represents a difference of 0 between the two paces.}
\label{Boston}
\end{figure}
While running the fastest race possible has been a goal for runners, it has also been a research topic for scientists for many years. Keller's model developed in the early 1970's, was the first to cast running an optimal race as an optimal control problem and has since been adapted and expanded by others such as Aftalion and Bonnans \cite{Aft}, Woodside \cite{Wood}, and Pitcher \cite{Pitch}. All of the results from these works show the importance of pacing, but lack attention to the different energy systems. Also, no runner model to date, considers in-race nutrition as an energy source available to the runner. In order to not run out of energy (and run the best they can) during a long distance run, a runner must consume food. There is a whole market of products for runners to use in order to deliver necessary fuel to the body quickly.\\
\\
%He used Newton's Second Law to derive  the differential equation for velocity; his model was built on Hill's work from 1927 in ``Dynamics of Sprint Running" \cite{Hill,Keller1}. He combined this with the idea of oxygen balance, which he expressed as how one's energy changes throughout the race as an energy differential equation.  Others, such as Woodside \cite{Wood}, as well as Aftalion and Bonnans \cite{Aft}, adapted Keller's system of differential equations \cite{Keller1} to account for fatigue and varying anaerobic factors respectively.\\
Our goal of this work was to build a more dynamic runner model that better represents the body's energy systems, including energy allocation of fat and carbohydrate energy dependent on velocity, as well as in-race nutrition. We also wanted to minmize the time it takes to run different length races (in particular the marathon distance race) by optimally choosing velocity and in-race nutrition consumption profiles through control of propulsion force and nutrition input. We first discuss the physics, biochemistry, and other factors that can limit or enhance a runner's performance before presenting our model, and optimization techniques. Results and various simulations are presented, followed by discussion.\\ 

\section{Background for Building Our Model}
\subsection{Background and groundwork from past models}
There are many different factors that coaches and runners have to think about in preparing for and running a race.  They have to consider all the training needed to get the runner in the best possible condition for the race as well as all the in-race components.  It takes time to train not just the muscles, but your body's energy pathways and metabolism. All the cumulative knowledge to develop the best training plan and best racing plan began with understanding how the body transfers stored energy into mechanical work. Originally, thermodynamics views exercise as the heat to mechanical energy transfer. This idea was applied to interpretations of the two known energy transfers: aerobic and anaerobic; however, biological energy transfer, or bioenergetics, was discovered to be better suited to describe the anaerobic and aerobic energy transfers due to chemical nature of the exchange.\\
\noindent Keller \cite{Keller1} considered Newton's Second Law and oxygen supply in his model, which was the first model of its kind, describing running races, and treats energy as available oxygen in the muscles per unit mass. His first equation is the equation of motion:
\begin{equation}\label{KellerV}
\frac{dV}{dt}=f - \frac{V}{\tau}
\end{equation}
where $t$ is time, $V(t)$ is the instantaneous velocity, $f(t)$ is the propulsive force per unit mass
% (equal to $bg$ in Hill's model \cite{Hill})
, $\frac{V}{\tau}$ is a resistive force per unit mass, and there is a constraint on the force, $0\leq f \leq f_{max}$.
His second equation governs energy, as the oxygen balance equation
 for $e_{an}$:
\begin{equation}\label{KellerE}
\frac{de_{an}}{dt}=\overline{\sigma}-fV
\end{equation}
where $\overline{\sigma}$ is the oxygen breathing and circulation rate in excess of that supplied in non-running state 
%The initial condition for this equation is: $e_{an}(0)=e_{an}^0$, and the anaerobic energy is under
with the constraint $e_{an}(t)\geq 0$  for all $0\leq t\leq T$ \cite{Keller1}. Solving the optimization problem and using only one energy type, Keller obtained results very close to world records. For example in the 100m race, Keller's theoretical results garnered a time of 10.07 seconds while the world record at the time was 9.9 seconds, a percent error of $1.7\%$ \cite{Keller1}. As both anaerobic and aerobic processes happen in the human body, modelers like Aftalion and Bonnans \cite{Aft}, as well as others \cite{Wood, Behncke}) included both processes in their model. Woodside \cite{Wood} added a fatigue factor for a long-distance runner model, but there is no consideration of what measures runners can take to combat this. We aim to address this and highlight the need to include 2 separate energies in a runner model. 
\subsection{Quantifiable Values that drive Performance}
When it comes to predicting one's running race performance, there are many ``calculator" tools that have been developed
%Jack Daniels, former Olympian, sports scientist and coach has one, as does McMillian, FirstBeat Analytics(Technology used in Garmin watches), the New Hanson group, and the Tinman group, just to name a few. 
%All of these calculators have many things in common; but in particular, they all heavily emphasize the runner's prior race performances. This is 
because the best indicator for a runner's future race performance is their past race performances, as is supported by decades of evidence in research \cite{Rob}.  Although there are new predictors being developed using machine learning that need lots of data, most runners have not run the number of races needed for such a robust data set. Thus, we chose to focus on other physical indicators that we can model through a system of equations that describe the physical and biochemical systems of a runner.\\
\\ 
Two quantifiable values that describe an athletes fitness are: VO2 max, and VLamax. Both play a role in how an athlete expends energy; and more precisely, what metabolic system they are primarily using at a given exertion level.   These values are comparisons of your aerobic and anaerobic systems. The relative strengths of these systems determine how well an athlete will perform at particular activities and can be used for good feedback in training.\\
%These values hold the most potential for personalizing calorie counts, training programs, and the feedback throughout training \cite{First}. A third physiological parameter that also affects a runner's performance is their running economy.  This can be described as the runner's efficiency at using their energy.  This can encompass their running form, the type of shoe they are wearing, amongst other things and can be race dependent due to conditions on race day. We will keep our focus on VO2 max and VLamax, but a runner's running economy does show up indirectly in this research.\\    
\\
Which metabolism is being used by the body is dependent on what percentage of one's current VO2 maximum (max) they are using, where VO2 max is the maximum amount of oxygen you can utilize during exercise.  The more fit you are, the higher your VO2 max will be. Note that VO2 max is measured in milliliters of oxygen consumed in one minute, per kilogram of body weight (mL/kg/min). This value ususally ranges between 20-60 with professional athletes holding values as high as 90 \cite{Anj}. One can certainly improve their VO2 max through purposeful training; although factors such as gender and genetics also play a role as the size of one's heart and their body fat percentage determine the body's capacity to carry oxygen through the blood. There are devices, formulas. and VO2 max calculators \cite{Kout,Daniels} that predict your VO2 max based on recent races and current level of activity.\\
\noindent In preparing for long distance races, runners are attempting to improve their VO2 max, as to allow the body to work at higher levels before needing to use the anaerobic system as the energy pathway. This means running at faster velocities without expending as much energy. When walking people are between $15\%$ and $30\%$ of their VO2 max, solely using their aerobic process, with the percentage increasing with exertion.
% and when they are jogging they are functioning between between $30\%$ and $50\%$ of their VO2 max, again using solely the aerobic process.
 When running a long-distance event, the runner mainly uses aerobic respiration and aims to stay at their aerobic threshold for the majority of the race.  One's aerobic threshold is at about 60\% VO2 max and is a level of intensity that can be sustained for a long time, where they are not using their anaerobic metabolism. Runners should only use their anaerobic metabolism for a small portion of the race, as operating at such an intensity cannot be sustained for very long.  The anaerobic system begins to be used between $75\%$ and $85\%$ of one's VO2 max.  During this time, the aerobic system is still used, but at a lower rate as the percentage of VO2 max increases up until $100\%$.  At that point, the anaerobic system is the main energy system with the aerobic system unable to be a significant contributor as the individual is functioning at a level where they have used up all available oxygen and need an energy system that does not require oxygen. \\
\\
A related, more tangible term to know is your velocity at your VO2 max, called VVO2max. It is perhaps a more valuable piece of data to know than just knowing your VO2max as it combines your aerobic capacity and running economy (or running mechanics), into one comparable number \cite{Billat}.
%Daniels extrapolated the velocity at VO2 max from the regression curve relating velocity and VO2 at sub-maximal levels \cite{Billat}.  When testing VVO2max in a laboratory test, the treadmill pace is increased and the testing administrator observes the runners VO2.
To test for a runner's VVO2max, on a treadmill, their pace is increased, and the administrator observes their VO2 increases linearly until it eventually starts to plateau.  When the VO2 plateaus, even as velocity is increased, no more oxygen is being taken in and the runner has reached their VVO2max. So, VVO2max is the minimum  running velocity with peak oxygen uptake and it's the speed at which there is the lowest contribution from the anaerobic metabolism, while also having maximum oxygen uptake.\\  
%Figure \ref{vvo2maxg} gives an example of this relationship.  When professionals administer a test to determine VVO2max, they obtain a figure similar to figure \ref{vvo2maxg}, where at first VO2 increases linearly with respect to velocity, but eventually levels off.  Where the graph levels off is the value at which VVO2max occurs \cite{McL}.   \\
\\
Before reaching VO2 max, runners hit their lactate threshold. Lactate is produced in anaerobic metabolism in the muscle whereas  carbon dioxide and water are created in aerobic metabolism. When the muscle lactate accumulates before its removal, the intensity has to be lessened as  the same level of intensity is unsustainable \cite{Scott}. Hitting one’s lactate threshold happens between 70-85$\%$ of their VO2 max. This threshold is a crossover point of the two energy processes where the body begins using anaerobic respiration for energy. When the anaerobic metabolism becomes the main energy process, the blood oxygen level is not enough  to keep up with the energy demand.\\

%Lactate from the muscle is converted back into glucose by the liver in a cycle called the Lactic acid cycle or the Cori cycle, which is named after named after Gertrude (Gerty) and Carl Cori. The glucose produced in the liver is released  into the blood, taken  by peripheral tissues (muscle, red blood cells, placenta, tumor), and converted into lactate again in anaerobic metabolism, which is released again into blood for another round of the cycle. Six ATPs, are used to produce glucose in the liver to deliver 2 ATPs to the peripheral tissues by the Cori cycle" \cite{EBC}. 

%It is beneficial to clear lactate effectively near threshold, which allows one to run faster before the lactate accumulation causes an issue. This is where the term VLa comes in.

\noindent  A runner's  VLa or lactate capacity, is the body's anaerobic power, or maximum ability to produce lactate (G Hillson, personal communication, April 21, 2021). The higher the VLa, the worse the runner is at clearing lactate near threshold. Marathon runners want to have a low VLa so that they can use more fat for energy and spare their carbohydrates (G Hillson, personal communication, April 21, 2021). \\
%Sprinters on the other hand, who are not concerned with sparing their glycogen energy, would want to have a higher VLa so that they have access to that faster metabolism (G Hillson, personal communication, April 21, 2021).\\
\\
While VO2 max is a good indicator of fitness, a runner's VLa max is what sets the professional runner apart from other runners. When looking for athletes for Nike's sub 2 project \cite{sub2}, they only looked at athletes who had a VO2 max greater than 65, but at that point VO2 max doesn't tell the whole story. A runner with a VO2 max of 65 or one with a VO2 max of 80, could have identical optimal races, depending on their VLa max.  It doesn't matter if a runner has a high VO2 if they aren't able to access all of that VO2 (G Hillson, personal communication, April 21, 2021). In conclusion, a long-distance runner wants to train their body to have a high VO2max (or VVO2max) and a low VLamax.

\subsection{Understanding the Body's Energy Sources}
The anaerobic metaboli-\\sm provides an athlete with quick energy, through the creation of glucose, but only 2 ATP (energy) molecules are obtained; wheras the aerobic metabolism creates 38 ATP molecules through a more length process.
%Past work, such as the hydraulic model described earlier used by Morton, Aftalion and Bonnans \cite{Aft,Morton} described the interaction of the two energy systems as the aerobic energy flowing into the anaerobic energy compartment, the only energy compartment the runner has access to.
 The aerobic and anaerobic systems occur in separate cellular compartments (mitochondria and cytoplasm respectively) and often at different rates, involve different reactants and products.
Not only is the allocation of the two separate energy processes of interest, but also which fuel is being utilized. Glucose and fatty acids provide most of the fuel required for energy production in skeletal muscles during aerobic exercise whereas glucose is the main source of energy in anaerobic exercise. The body has significantly more energy available in the form of fat, but the rate of using this energy form can not be increased at high exercise intensities when the anaerobic metabolism is the main mechanism. Thus, the body is mainly able to use fatty acids as an energy source at low levels of intensity \cite{biochem}. When you are not exercising, approximately 30 percent of your energy comes from glycogen and 70 percent from fat stores \cite{biochem}. These percentages shift when intensity increases, as does the number of calories being burned. Glucose is preferred as it is readily available and quick metabolized, but is limited.\\
%Lipolysis in the muscle tissue breaks down triglyceride in lipoproteins (fat in the blood) and stored in the fat depots in the muscle to release fatty acids. Lipolysis also occurs in adipose cells and release free fatty acids into blood circulation.  In the muscle cells, fatty acids are transported into the mitochondria and oxidized via beta oxidation to yield acetyl CoA, which is combined with oxaloacetate to form citrate that is used in the Krebs cycle.
%The contribution of fatty acids as the dominant fuel peaks at low to moderate exercise levels at approximately 60$\%$ VO2 max. The VO2 max at which fatty acid oxidation peaks is between 45$\%$ and 65$\%$\cite{biochem}. With the elevation of exercising intensity, glucose used in anaerobic metabolism becomes the only increasable source for energy production during medium to high intensities. When glycogen stores are limited, the body will eventually be forced to stop performing at this intensity. During low intensity exercise, an intensity comparable to walking, the rate of appearance of fatty acid in plasma matches closely the rate of fatty acid oxidation. Most of the energy requirements can be met from oxidation of fatty acid derived from the blood, with a small contribution from plasma glucose. Even when low intensity exercise is sustained for one to two hours, the pattern of fuel utilization does not change considerably. Presumably, this is because the muscle energy requirements can be met almost exclusively from the oxidation of the fatty acid mobilized from the large fat (adipose) stores. 

\noindent Figure \ref{areafig} is a diagram showing fuel utilization between fat and glycogen as a function of percent VO2 max. In Figure \ref{areafig}, at low percent of VO2 max, the fuel utilization is low, and the percentage contribution from fat is substantial. With the rise of VO2 max, when running faster, the rate of fuel utilization increases and the percent contribution of fat decreases. 
%\vspace{-3.2cm}
\begin{figure}[htbp]
\vspace{-3cm}
  \centering
    \includegraphics[scale=.4]{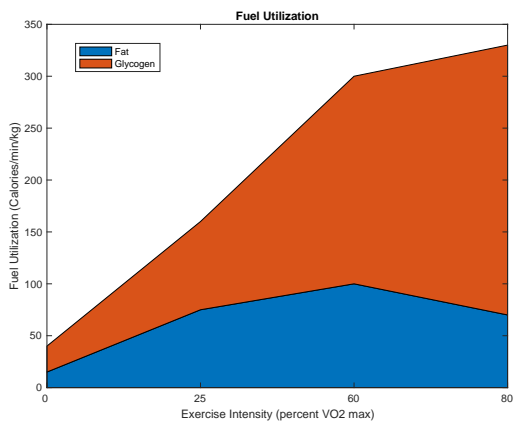} \\
    \vspace{-2.6cm}
\caption{Fuel Utilization: allocation of calories per minute per kilogram used as a function of percent VO2 max (adapted from \cite{biochem}}
\label{areafig}
\end{figure}
\subsection{Why We need Nutrition during the Race}
When running races such as the marathon one must consider pacing strategies. The runner must not only consider oxygen availability, but how much energy the body has stored.  To optimally run these races, runners are mainly using the aerobic system, but their intensity with corresponding percentage of the VO2 max is high enough for the body to use glucose as energy, causing them to still burn through the limited supply of glycogen.  Runners of different levels, masses, running at relatively different VO2 values, burn through this glycogen at different rates; however, it is commonly accepted that on average runners burn just over 100 kilocalories per mile, meaning their stores will be depleted after about 20 miles or between 1.5 hours and  hours \cite{biochem,Murray}.  When these stores are depleted, the term `bonking' is commonly used, where a runner must slow down significantly or forced to stop running altogether and walk.  As only one person to date has completed the marathon in under 2 hours, this is a major problem for long distance runners that must combat. 
%The only way to have enough energy to finish the race is to somehow increase one's energy store.  This can only be done through consumption.  We know so much more about how the body handles stress and energy delivery than ever before.
In the 1960's, work by \cite{Burke} confirmed that blood glucose concentrations were linked to fatigue and that eating hard candies during a race prevented weakness and fatigue during a race \cite{Murray}.  
%This was not scientifically confirmed until the 1960's when a group of Scandinavian scientists took muscle biopsies to examine fuel utilization and enzyme activity in the muscle \cite{Burke}. Since then, there has been vast research to determine what the runner can do with regards to nutrition during the race to improve their performance and avoid fatigue.  A whole industry has been developed just for in-race nutrition.\\
Carbohydrates are digested in the small intestine and converted into glucose.  
%The small intestine absorbs the glucose which is transported into the bloodstream where it travels via the portal vein to the muscles, brain, liver, or stays in circulation.  
Glucose is stored mainly in the muscles and the liver as glycogen, a chain of multiple glucose residues, but is also available for
%\newpage
%\begin{figure}[h!]
%
%\begin{center}
%\hspace{-1.5cm}
%    \includegraphics[scale=.6]{mydrawnbodysystem} \\
%    \vspace{-1.5cm}
%\end{center}
%\caption{Diagram of Carbohydrates: From Intake to Muscles}
%\label{glycogenintake}
%\end{figure}
%\newpage
immediate use if necessary \cite{biochem}. The body can store about 600 grams of glycogen, with 500 grams stored in the muscles and 100 grams stored in the liver \cite{biochem}, totaling 2400 Kilocalories of glycogen stores in the body. 
\\
\\
During a long distance running race when glycogen stores are depleted, ingested or exogenous carbohydrates are quick sources of energy for the muscles, available once absorbed by the muscles from the blood. Runners typically consume gels, which are a concentrated dose of glucose throughout long distances races.  Most runners solely consume gels, that contain approximately 25 grams of carbohydrates (100 Kilocalorie). 
%These gels come in packages that a runner squeezes into their mouth over a short time frame.  Note that while most of these carbohydrates are available to the muscles for use, there is a small percentage used for basic bodily functions, as is the case with any food energy.  
Taking in nutrition during the race allows a runner to move longer before glycogen stores are depleted and they've reached some anaerobic energy threshold where they must walk. It is known that the muscles absorb plasma glucose at a maximal rate of 1g/min-1.7 g/min \cite{biochem, Wallis,jentjens} depending on the sugar mixture. This means that while it may only take 3-5 minutes for some of the carbohydrates from a 100 calorie intake to reach your muscles, it can take approximately $25$ minutes for all of the carbohydrates from the package to be absorbed. While there is no limit to how many carbohydrates a runner can ingest, too many carbohydrates consumed in a short time will result in digestive discomfort, forcing the runner to slow down.   
\section{Our Model}
We build on the current models by adding in some novel terms as well as completely reformulating the energy equation. Our overall goal is to determine the best nutrition and pacing strategy to use when running a marathon (or other long distance race) in order to finish in the shortest amount of time.  Our first objective, is to develop a runner model that takes into account fuel allocation depending on percentage VO2 max, fuel intake (in-race nutrition), as well as the force which is applied to the ground by the runner, by using a system of differential equations. Next, we would like to cast this as an optimal control problem to determine the optimal velocity and in race nutrition consumption profiles through control of propulsion force and nutrition input.\\
%\\ 
%Both nutrition and metabolism differentiation dependent on an individual's VO2 max will be considered in our model to more accurately account for how the body metabolizes energy, and what runners have available to them during races.  Through interpreting energy exchange in this manner, we want to design a model that is understandable and used by the general public.  As one can easily compute their VO2 max, it would not be difficult for people to then use results from our model to create a race strategy.   It is clear that one would want to have their glycogen stores completely full to begin a race, but when to take nutrition during the race, how often to take nutrition, and how much to take are questions to address here.  We will assume that the runner is not restricted by when they can take the nutrition, as is the case for most runners using aide stations for their nutrition, but implore race directors to consider this model when setting up their course.\\
\\
Our system of ordinary differential equations for: $V,\text{Velocity,}\ \ E_F,\text{Fat Energy,}\\ 
E_G,\text{Glycogen Energy,}\ \ N,\text{Nutrition}$\ ,where $V$ will be measured in meters/min, $E_F$ and $E_G$ in KJ/Kg, $N$ in KJ, and our control force $f(t)$ in meters/min$^2$, are given by:
\begin{align}
\frac{dV}{dt}=\ &f(t) - \frac{V}{\tau}\label{myV}\\% -\frac{\max(20-E_G,0)}{100*\sqrt{E_G^2}+1}\\
\frac{dE_G}{dt}=\ &c_3j(N)-af(t)Vglyc(V)\label{myEG}\\
\frac{dE_F}{dt}=\ &-af(t)V(1-glyc(V))\label{myEF}\\
\frac{dN}{dt}=\ &s(t)-dN-j(N)\label{myN} \\
\vspace{-1cm}
\nonumber \end{align} where $a=\frac{1}{1000}$.  Figure \ref{ogsys} is a diagram of our system, including the velocity compartment.  Velocity is not directly connected to any of the other compartments via energy transfer, but its impact on the two energy compartments can be seen in equations (\ref{myEG}) and (\ref{myEF}). Our nutrition intake strategy is the fuel source function, $s$, entering in the nutrition compartment.
%\newpage
\begin{figure}[]
\vspace{-1.5cm}
\begin{center}
    \includegraphics[scale=.37]{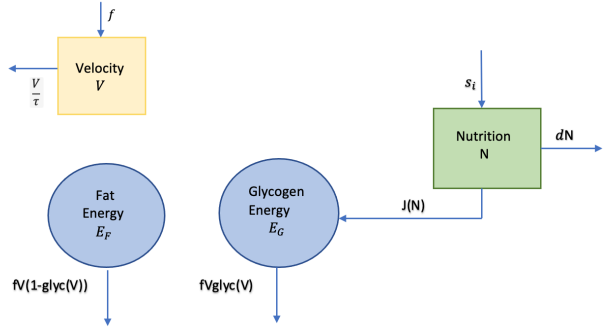} \\
%    \vspace{-2cm}
\end{center}
\vspace{-1.5cm}
\caption{Energy-Nutrition-Velocity Diagram}
\label{ogsys}
\end{figure}
%\newpage
In equation (\ref{myV}), there is a propulsive force per unit mass $f(t)$, and $|\frac{V(t)}{\tau}|$ is a resistive force per unit mass. The initial condition for this equation is: $V(0)=0$, as we start from rest, and we have a constraint on the force, $0\leq f(t) \leq f_{max}$.  This is the same equation that Hill, Keller, and others used for velocity \cite{Hill,Keller2,Wood,Behncke,Aft}.\\
\\
As energy is available from two different processes in the body, anaerobic and aerobic metabolisms, we consider two energy sources.  In aerobic metabolism, the body uses energy from both fat and glycogen, while in anaerobic metabolism only energy from glycogen is used, but at a faster rate. Thus, our total energy, $E$, can be written as $E= E_F+E_G$.
 %where $E_F$ is energy from fat and $E_G$ is energy from glycogen.
%$$\frac{dE}{dt}=\frac{dE_F}{dt}+\frac{dE_G}{dt}\ \ .$$\\
The human body has large stores of fat for use as energy, but the body prefers to use glycogen when moving at higher rates; thus, a higher percentage of fuel usage comes from fat at lower velocities. As runners typically take gels that are mainly sugar (not a significant source of fat), we assume there is no input into the fat energy compartment. The body expends energy at a rate, $f(t)V$ (work done). As we are allocating the energy usage between the two sources, our fat energy differential equation (\ref{myEF}), and glycogen energy differential equation (\ref{myEG})
%: $$\frac{dE_F}{dt}=-afV(1-glyc(V))$$
have glycolitic function, $glyc(V)$, a fuel allocation function of one's glycogen energy through the anaerobic and aerobic pathways dependent on the quotient of one's instantaneous velocity, $V$, and their velocity, $VVO2max$, at 100$\%$ VO2max, written $\frac{V}{VVO2max}$. Thus, $1-glyc(V)$ accounts for the fuel allocation function of one's fat energy through the aerobic pathway dependent on velocity compared to one's velocity at 100$\%$ VO2max.\\
\\
It is difficult to know one's current percentage VO2 max; however, one could know a priori their velocity at 100 percent VO2 max.  Billat and Koralsztein \cite{Billat} showed that one's VO2 max and their velocity at VO2 max are linearly related.  As velocity is more easily computable we will use this relationship in our model in determining the allocation of energy at different velocities compared with one's velocity at VO2 max. To approximate $glyc(V)$, we first described it it as a piecewise continuous function on intervals of $\frac{V}{vvmax}$ based on \ref{areafig} from \cite{biochem} and then approximated that function by a function that smooths the points where the derivatives do no exist (described in detail later).
Equations (\ref{myEG}) and (\ref{myEF}) have a convex combination of this function and the fat allocation function of $1-glyc(V)$, as the total work rate must equal $f(t)V$. We obtained the piecewise graph for $glyc(V)$ from Figure \ref{areafig} as well as from the literature \cite{biochem}.  Figure \ref{areafig} gives a good estimate of how the body uses fat energy versus glycogen energy. 
% It does not detail allocation after $80\%$ VO2 max.
  We assume that the anaerobic pathway is not used until after $60\%$ VO2 max, and that the energy from fat linearly decreases to 0 by $100\%$ VO2 max at which point the anaerobic system, and thus the glycogen compartment, is solely used. Figure \ref{3glyc} is a graph of three different possible $glyc(V)$ functions that we obtained from our Figure \ref{areafig}.\\  
%Vector points obtained from our piecewise continuous approximation are shown as well as the continuous, twice differentiable, cubic spline function approximation.\\
 %  Also, figure \ref{fats} shows the glycogen usage dependent on percent VO2 max as well as the fat usage dependent on percent VO2 max on the same graph.\\
\\
Recall, that while what percentage of one's VO2 max at which they are running is explicitly in the $glyc(V)$ function, VO2 max is not the only value that effects our fuel utilization.  Lactate capacity (VLa), also impacts one's fuel utilization as described earlier. We account for differences in VLa, by considering different structures for our fuel utilization function, $glyc(V)$.  The function is still dependent on percent VO2max, but varies in fuel utilization at particular percent VO2 max values. Figure \ref{3glyc} shows three different $glyc(V)$ functions corresponding to 3 different lactate capacities.  We use these 3 different glyc structures to represent runners who have a low VLa (good), an average VLa (avg), and a high VLa (bad).  The runners with a low VLa (good) are able to run at a faster pace than those with a higher VLa without accumulating as much lactate in their muscles. The VLa is a feature that sets the best professionals apart from one another.\\
\\
In equation (\ref{myEG}) for glycogen energy there is a source term with $j(N)$ that comes from the nutrition differential equation (\ref{myN}).  Our $j(N)$ term is a nutrition consumption function increasing energy available in the muscles in the form of glycogen, but at a bounded rate: $j(N)=c_4N$  with rate constant $c_4= \frac{1}{m}$, the inverse of the runners mass. Our initial conditions for our two energy equations are $E_G=144$ Kilojoules per unit mass, assuming the runner has full glycogen stores, and $E_F(0)=3439$ Kilojoules per unit mass, dependent on the runner's body fat percentage. Both of the energies must stay non-negative, and dictate the choices of $f(t)$ and the corresponding V. Thus, in particular, if $E_G=0$, $f(t)=0$. This is not optimal and would be avoided in an optimal race until the very end. 
% We obtained these initial conditions by assuming that one's glycogen stores are completely full and took the $E_F(0)$ from literature, assuming an body fat percent composition. 
%meaning that there are approximately $2000$ Kilocalories$\approx 8400$ Kilojoules worth of energy in that ``compartment" and took that per unit mass as the rest of the system has those units.  We obtained the $3439$ Kilojoules of energy per unit mass for $E_F(0)$ 
%from the literature, which assumes a certain, measurable, body fat percent composition. It should also be noted that the force $f(t)$ must be chosen such that the energy compoents stay non-negative.
% An example of someone with this amount of fat energy could be $182$ lb runner with $9\%$ body fat.  
%Both of these initial conditions would be dependent on the runner.
%; although, as a runner's ability to store glucose is harder to measure, and this number doesn't vary as much per individual, $E_G(0)$ may not change much.
\begin{figure}[htbp]
\begin{center}
\vspace{-2cm}
    \includegraphics[scale=.4]{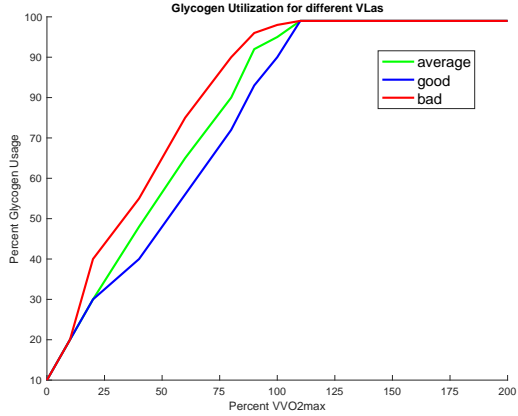} \\
    \vspace{-2.5cm}
\end{center}
\caption{Three different glyc functions corresponding to runners with a good, average, or bad VLa.} 
\label{3glyc}
\end{figure}
%We considered type I and II functional responses for this function.  Functional responses were determined by Holling that model ``the response of consumption of prey by individual predators to changes of prey density" \cite{Holling}. Type II responses center around the concept of handling time associated with each prey item eaten/digested/attacked. In a type II functional response, at low prey densities a smaller proportion of a predators time is spent handling/digesting the prey, but as prey density increases, the proportion of time spent handling/digesting increases until it reaches some maximum value, where it levels off.  Ingestion of food is often described by this type of functional response due to the decelerating intake rate seen in a type II functional response, and the consumers limited rate to process food \cite{Mulder,Holling}. However, we found that a simpler type I function response is more appropriate for our model, where our intake rate increases linearly until it reaches the maximum intake rate.  Since we do not believe the body's intake rate decreases as it approaches its maximum, this type of response is more realistic than a type II functional response. 
%The nutrition (energy) that increases glycogen energy is modeled via the differential equation\\
%$$ \frac{dN}{dt}=s(t)-dN-c_3j(N)$$ \\
Further, in our nutrition energy differential equation, (\ref{myN}), $s(t)$ is a source term from nutrition input such as a gel.  Each gel is roughly 100 Kilocalories (4.18 Kilojoules). The addition of the source term slows the rate at which $E_G$ decerases.
% Thus, for our source function we simply input in the nutrition at different time points according to determined nutrition strategies.\\
%\begin{align}\label{s}
% s(t)&=418rectpuls(t-(w+h(0)),p)+418rectpuls(t-(w+h(1)),p)...\\
% &+418rectpuls(t-(w+h(n)),p)\nonumber \ \ .
% \end{align} 
%We use a rectpuls function to describe the consumption of the gel over a short time period.  A rectpuls function is a function form in MATLAB that returns a continuous, periodic, unit-height rectangular pulse at sample times supplied by the user.  The parameter $w$ represent the amount of the first gel taken, $h$ represents the frequency of consumption, and $p$ represents the time over which the gels are consumed. In this example, that goes with Figure \ref{rectpuls}, the function represents a 100 Kilocalorie gel being taken every 2000 seconds, starting at $t=2000$s in the figure.  Both, the first consumption time, and the frequency of consumption are something we plan to vary in different scenarios in our optimization problem.  In Figure \ref{rectpuls} we have $p=2$ seconds, as a runner consumes gels quickly.  One can also see in this Figure that the time goes out to 12000 seconds,  3 hour and 20 minute, a reasonable marathon finishing time for a trained runner.  
The $-dN$ term in the nutrition differential equation represents the nutrition used for basic bodily function, not available to the muscles for energy usage.  The initial condition for this equation is $N(0)=0$ as we assume there is nothing in this compartment when the race starts. Model parameters can be found in the results section in Table \ref{table:kipchoge}.\\

\section{Optimal Control}
Our second goal is to determine the best strategy to use when running a long distance race in order to finish in the shortest amount of time.  Casting this as an optimal control problem, we could think of this as a problem of maximizing the distance over a fixed time interval or equivalently (proved in \cite{Aft}), as a problem of minimizing time over a fixed distance. We have two controls in our problem: propulsive force, $f(t)$, and fuel intake $s(t)$.
% We are considering the impact from fuel intake as an optimization over several different intake scenarios. 
 We will determine the optimal control $f(t)$ for each of the intake strategies and then optimize over those intake strategies.
If we chose to solve the minimum time problem. an extra isoperimetric constraint must be considered.  Thus, we choose to solve the maximum distance problem.  
\subsection{Optimal Control Problem Setup}
For the maximum distance problem, assuming we have a nutrition strategy, $s_i$, we maximize the objective functional:\\
  \begin{equation}\label{myOpt}
J_i(f)=\int_0^TV(t)dt
\end{equation}
for $s_i\in S$ (a finite set of nutrition strategies), where $$S=\{s_1,s_2,...,s_N\in L^2(0,T)| 0\leq s_i(t)\leq s_{max}, a.e., i=1,2,...,N\}$$ we determine an optimal control depending on $s_i$, and then optimize over our set $S$. We obtain continuous velocity and energy profiles corresponding to our control force and then optimize over discrete $s_i$.  We solve the following with fixed $T$:
\begin{equation}\label{eqtnU}
\max_{S}\max_{A_i}J(s_i,f)=\max_{S}\max_{A_i}\int_0^TV(t)dt
\end{equation}
 with bounded controls:\ \ $0\leq f\leq f_{max}$, $0\leq s_i(t)\leq s(t)_{max}$, and control set $$U_i= s_i \times \{f \in L^2 (0, T) \, | \, 0 \leq f(t) \leq f_{max}, a.e.\}$$
subject to state constraints:\ \ $0\leq E_F(t)$\ \ ,\ \ $0\leq E_G(t)$, for $t \in [0,T]$\\
initial conditions:\ \ $V(0)=0$,\ \ $E_F(0)=K$,\ \ $E_G(0)=144$,\ \ $N(0)=0$,\\
and state equations (\ref{myV})-(\ref{myN}).
%\begin{align*}
%V'&=f(t)-\frac{V}{\tau}\\%-\frac{\max(20-E_G,0)}{100*\sqrt{E_G^2}+1}\\
%&\\
%E_F'&=-afV(1-glyc(V))\\
%&\\
%E_G'&= c_3j(N)-afVglyc(V)\\
%&\\
%N'&=s(t)-dN-j(N)
%\end{align*}
%where $a=\frac{1}{1000}$, $\tau >0$, $d>0$, and $c_3>0$.
Our admissible control set for $s_i \in S$  is $$A_i = \{ (s_i, f) \in U_i \, | \,  E_F(t) \geq 0, E_G(t) \geq 0,\ \text{for} \, t \geq 0 \}$$
%Recall that our glycogen fuel allocation function, $glyc(V)$ is written as:\\
%\begin{align*}
%glyc(V) = \left\{ \begin{array}{ccc} 
%               0.1 & for & 0\leq \frac{V}{vvmax}\leq 0.20 \\
%                \ & \ &\\
%                2(\frac{V}{vvmax})-0.30  & for & 0.2\leq \frac{V}{vvmax}\leq 0.3 \\
%                    \ & \ &\\
%                1.8(\frac{V}{vvmax})-0.24  & for & 0.30\leq \frac{V}{vvmax}\leq 0.40 \\
%                    \ & \ &\\
%                 0.4(\frac{V}{vvmax})+0.32   & for & 0.40\leq \frac{V}{vvmax}\leq 0.50 \\
%                    \ & \ &\\
%               0.5(\frac{V}{vvmax})+0.27   & for & 0.50\leq \frac{V}{vvmax}\leq 0.60 \\
%                    \ & \ &\\
%                0.8(\frac{V}{vvmax})+.09 & for & 0.60\leq \frac{V}{vvmax}\leq 0.7\\
%                \ & \ &\\
%                  0.9(\frac{V}{vvmax})+.02   & for & 0.70\leq \frac{V}{vvmax}\leq 0.80 \\
%                    \ & \ &\\
%                 1.05(\frac{V}{vvmax})-0.10   & for & 0.80\leq \frac{V}{vvmax}\leq 1.0 \\
%                    \ & \ &\\
%               0.5(\frac{V}{vvmax})+.45   & for & 1.0\leq \frac{V}{vvmax}\leq 1.1 \\
%                    \ & \ &\\
%               	1.0   & for & 1.1\leq \frac{V}{vvmax}\leq 2 \\
%                    \ & \ &\\
%                \end{array} \right.
%\end{align*}
%and the function $j(N)$ is written as:\\
% $$j(N)=c_4N$$\\
%with limiting rate constant $c_4$.\\
Note that $V(t)$ and $N(t)$ are non-negative, from (\ref{myV},\ref{myN}).
% Also, due to our admissible control definition, the energy states of $E_F(t)$ and $E_G(t)$ are also non-negative. 
%Also, in its current form from equation (\ref{glyc}), also seen above, $glyc(V)$ is continuous, but is not smooth.  For our analysis and approximations, we chose to use a smoothed version of $glyc(V)$, by taking a vector of sample points from the piecewise $glyc(V)$, and creating a cubic spline function so that this allocation function is differentiable over its domain.\\ 
\subsection{Existence of the optimal control}

\begin{theorem} 
Given a nutrition strategy $s_i$, there exist an optimal control $f^*$ that solves the optimal control problem of maximizing (\ref{myOpt}) with corresponding state equations (\ref{myV})-(\ref{myN}) over admissible controls, $$\mathcal{U}=\{f\in L^2 (0, T) \, | \, 0 \leq f(t) \leq f_{max}, a.e., \, E_F(t)\geq 0,\ E_G(t)\geq 0,\ \text{for}\ t\geq0\}.$$ 
\end{theorem} 
\begin{proof}
First, the admissible control  set is nonempty since the control $f=0$ satisfies the conditions. %We also reiterate that our states and control force are all non-negative. 
 Our objective functional, $J(f)$ is  uniformly bounded due to bounds on the states and controls.  
% %As the problem is over a finite time and the controls $f$ are uniformly bounded, the states are uniformly bounded.%, (in particular, $V(t)$ is uniformly bounded), and $glyc(V)$ is uniformly bounded (continuous over a finite time). 
%Thus, $J(f)<CT$ for all admissible $f$, which implies that
 There exists a maximizing sequence of control functions denoted by ${f^n}$ such that $$\lim_{n\rightarrow \infty}J(f^n)=\sup_{f\in \mathcal{U}} J(f),$$  
 and  a corresponding sequence of states $V^n,E_{F}^n,E_{G}^n,N^n$.  Note that the control sequence and the state sequences are uniformly bounded.
 % The control sequence is uniformly bounded in $L^2(0,T)$. 
 By the Banach-Alaoglu theorem \cite{Fried}, there exists a control $f^*$ such that on a subsequence $f^n\rightharpoonup f^*$ weakly in $L^2(0,T)$. 
%and using lower semicontinuity of $L^2$ norms with respect to weak convergence we obtain $$\int_0^T[f^*(t)]^2dt\leq \liminf_{n\ra \infty}\int_0^T[f^n(t)]^2dt\ \ .$$ 
From the state differential equations we see that the derivative sequences are also uniformly bounded. Thus, $V^n,E_{F}^n,E_{G}^n,N^n$ are uniformly Lipschitz and thus equicontinuous.  Therefore by the Arzela-Ascoli theorem, on a subsequence, %f^n\rightharpoonup f^*$ in $L^2(0,T)$ , 
we have uniform convergence of our corresponding state functions $$V^n\rightarrow V^*, E_{F}^n\rightarrow E_F^*, E_{G}^n\rightarrow E_G^*, N^n\rightarrow N^*\ \ .$$ We need to show that $f^*$ is an optimal control and $V^*,E_F^*,E_G^*,N^*$ are corresponding optimal states. To show those states correspond to control to $f^*$, we illustrate this by showing the $E$ differential equation: %First, as we can write our velocity state equation as $$V^n(t)-V_0=\int_0^tf^n-\frac{V^n}{\tau}ds$$ as well as the others in a similar manner, in order to show that the states, $V^*$,$E_F^*$, $E_G^*$, and $N^*$ are the states corresponding to $f^*$.  %To do so, we need to show the convergence of the integral of the state differential equation sequence to the integral of the optimal state differential equation for each of our 4 state differential equations. 
%
%For the $V$ differential equation,% $V'=f-\frac{V}{\tau}$, 
%we have a sequence of functions $f^n\rightharpoonup f^*$ weakly in $L^2$ for the first term and a uniform convergent sequence in the second term. By the definition of weak convergence in $L^2$ and uniform convergence of the state,  $V^*$ satisfies the state differential equation corresponding to $f^*$. The $E_F^n$ sequence satisfies 
$$E_F^n(t)-E_F(0)=\frac{-1}{1000}\int_0^tf^nV^n(1-glyc(V^n))ds$$ 
%Thus, we need to show $$-\int_0^t\frac{f^nV^n}{1000}(1-glyc(V^n))ds\rightarrow -\int_0^t\frac{f^*V^*}{1000}(1-glyc(V^*))ds\ \ .$$\\
%We have a sequence converging weakly in $L^2$ ($f^n\rightarrow f^*$) and a sequence converging uniformly: $$\frac{V^n}{1000}(1-glyc(V^n))\rightarrow \frac{V^*}{1000}(1-glyc(V^*))$$.  We will show that this product converges since the sequences are uniformly $L^{\infty}$ bounded.\\ 
%First, note that as $V^n$ converges uniformly and $glyc(V^n)$ is differentiable over its domain, $glyc(V^n)$ is uniformly Lipschitz, and hence $\frac{1}{1000}(1-glyc(V^n))$ converges uniformly. So,\\
% $$\int_0^t\frac{f^nV^n}{1000}(1-glyc(V^n))ds\rightarrow \int_0^t\frac{f^*V^*}{1000}(1-glyc(V^*))ds$$ or written as $$|\int_0^t\frac{f^nV^n}{1000}(1-glyc(V^n))ds-\int_0^t\frac{f^*V^*}{1000}(1-glyc(V^*))ds|\rightarrow 0$$
%We write,
and
\begin{align*}
&|\int_0^t\frac{f^nV^n}{1000}(1-glyc(V^n))ds-\int_0^t\frac{f^*V^*}{1000}(1-glyc(V^*))ds|\\
%=\ \ &|\int_0^t\frac{f^nV^n}{1000}(1-glyc(V^n))ds-\int_0^t\frac{f^nV^*}{1000}(1-glyc(V^*))ds\\
%+\ \ &\int_0^t\frac{f^nV^*}{1000}(1-glyc(V^*))ds-\int_0^t\frac{f^*V^*}{1000}(1-glyc(V^*))ds|\\
%&\\
%\leq\ \ &|\int_0^t\frac{f^nV^n}{1000}(1-glyc(V^n))ds-\int_0^t\frac{f^nV^*}{1000}(1-glyc(V^*))ds|\\
%+\ \ &|\int_0^t\frac{f^nV^*}{1000}(1-glyc(V^*))ds-\int_0^t\frac{f^nV^*}{1000}(1-glyc(V^*))ds|\ \ \ \ (\text{by the triangle inequality})\\
%&\\
%=\ \ &|\int_0^t\frac{f^n}{1000}(V^n(1-glyc(V^n))-V^*(1-glyc(V^*)))ds|+|\int_0^t\frac{V^*}{1000}(1-glyc(V^*))(f^n-f^*)ds|\\
\leq\ \ & ||f^n||_2||V^n-V^*||_2+|\int_0^tV^*(f^n-f^*)ds|\ \ \ \ (\text{by Holder's Inequality})\\
\end{align*}
 %$f^n$ sequence is uniformly bounded in $L^2$ and as
% Thus, as $V^n\rightarrow V^*$ uniformly, and by
% % the first term goes to 0. As $V^*(1-glyc(V^*))$ is a fixed bounded function and by the definition of
%weak convergence in $L^2$ of the control sequence,
%% Thus, our second term also goes to 0. 
%$$-\int_0^t\frac{f^nV^n}{1000}(1-glyc(V^n))ds\rightarrow -\int_0^t\frac{f^*V^*}{1000}(1-glyc(V^*))ds\ \ .$$  All terms converge and we can conclude: $$ E_{F}^n\rightarrow E_F^*.$$
% As our $E_G$ differential equation, (\ref{myEG}), has the same product structure as seen in the $E_F$ differential equation with the addition of another term with $N^n$ that converges uniformly, $E_G^*$, $E_F^*$, and $N^*$ satisfy the differential equation corresponding to $f^*$.
% % we obtain $$E_{G}^n\rightarrow E_G^*.$$ Thus, $E_G^*,E_F^*$ satisfy the state differential equations corresponding to $f^*$. Lastly, as our sequence $N^n$ is uniformly convergent,  we can pass to the limit in equation (\ref{myN})  and thus $N^*$ satisfies the state differential equation corresponding to $f^*$ and 
 We have that all 4 states correspond to the control $f^*$\ . 
Next, $f^*$ is an optimal control since, 
\begin{align*}
%\sup_{f\in U_2}J(f)&=
\lim_{n\rightarrow \infty}J(f^n)&=%\ \ \ \ (\text{as we chose}\ f^n\ \text{to be a maximizing %sequence})\\
\lim_{n\rightarrow \infty}\int_0^TV^ndt  %\ (\text{by definition of}\ J)\\
=\int_0^TV^*dt %\ \ \ \ (\text{from the convergence of}\ V^n\rightarrow V^*)\\
= J(f^*)\ .
\end{align*}
%\noindent Thus, $f^*$ is an optimal control and we can say $$\sup_{f\in U_2}J(f)=J(f^*)=\max_{f\in U_2}J(f)\ \ .$$\\
\end{proof}
Now that we've proven an optimal control exists, we begin solving the optimal control problem using Pontryagin's Maximum Principle (PMP) \cite{Pontr} by determining the set of necessary conditions, using the Hamiltonian with state constraints. We solve the maximum distance problem, obtaining an $f^*$ for each nutrition strategy, $s_i$, and maximize over our set of finite strategies $S$.
\subsection{Necessary Conditions for the Maximum Distance Problem}
We determine the necessary conditions by using PMP \cite{Pontr} to compute the Lagrangian, which is the Hamiltonian with state constraints:
\begin{align}
\mathcal{L}&=V+\lambda_1(f(t)-\frac{V}{\tau})+\lambda_2(-af(t)V(1-glyc(V)))+\lambda_3(c_3j(N)-af(t)Vglyc(V))\label{Lagr} \nonumber \\
&+\lambda_4(s(t)-dN-j(N))+\eta E_G\nonumber \ ,
\end{align}
 where the penalty function $\eta(t) \geq 0$ is a Lagrangian multiplier appended to $E_G$ as a state constraint penalty. Note that we do not have a state constraint penalty for $E_F$ as $E_F$ will not get close to $0$ over the time interval. The function $\eta$ satisfies: $\eta \equiv 0$ where $E_G^*>0$, and $\eta \geq 0$ otherwise when the state constraint is tight (i.e. $E_G=0)$  Since we seek to maximize $\mathcal{L}$ with respect to $f(t)$, a state variable violating the constraint would decrease $\mathcal{L}$ and not be optimal as $\eta E_G < 0$.\\
 \\
 Also, PMP\cite{Pontr} gives the existence of adjoint functions , $\lambda_1$, $\lambda_2$,$\lambda_3$ and $\lambda_4$, satisfying  system  (\ref{lam1}) (\ref{lam4}) appending the state equations to the Hamiltonian.
\begin{align}
\lambda_1'=-\frac{\partial \mathcal{L}}{\partial V}&= -(1-\frac{\lambda_1}{\tau}-\lambda_2af(t)+\lambda_2af(t)V\frac{\partial glyc}{\partial V}+\lambda_2f(t)aglyc(V)-\label{lam1}\\
&\lambda_3af(t)V\frac{\partial glyc}{\partial V}-\lambda_3af(t)glyc(V))\nonumber\\
&\nonumber \\
\lambda_2'=-\frac{\partial \mathcal{L}}{\partial E_F} &= 0\ \ , \ \ \lambda_3'=-\frac{\partial \mathcal{L}}{\partial E_G} = -\eta\\
 &\nonumber\\
\lambda_4'=-\frac{\partial \mathcal{L}}{\partial N} &= -(c_3c_4\lambda_3-\lambda_4d-\lambda_4c_4)\label{lam4}
\end{align}
with transversality conditions $\lambda_1(T)=0$, $\lambda_2(T)=0$, $\lambda_3(T)=0$, $\lambda_4(T)=0$\\
\\
 For $t\in [0,T]$, due to it's initial condition, $E_F$, stays positive for any choices of $f(t)$ in the admissible control set and thus $\lambda_2=0$ over $[0,T]$.
Since our objective functional and state equations are linear in the control, we consider the different signs of the derivative of $\mathcal{L}$ with respect to $f(t)$, which is the switching function  $\phi(t)$.
%\begin{equation}
%\phi(t)= \frac{\partial \mathcal{L}}{\partial f}=\lambda_1-\lambda_2aV+\lambda_2aVglyc(V)-\lambda_3aVglyc(V)\label{phi}\ \ .
%\end{equation}
\begin{equation}
\phi(t)= \frac{\partial \mathcal{L}}{\partial f}=\lambda_1-\lambda_3aVglyc(V)\label{phi}
\end{equation}
At time $t$, for $\phi(t)<0$ we will have a maximum when $f^*(t)=0$, whereas for $\phi(t)>0$ we will have a maximum when $f^*(t)=f_{max}$. However, PMP \cite{Pontr} does not tell us what happens when $\phi(t)=0$.  More information with respect to control $f$ than $\frac{\partial \mathcal{L}}{\partial f} = 0$ is needed when $\phi(t)=0$ on a subinterval.  If this happens only for for finite number of time points, we would have a bang-bang control and those points would represent the switching times. On the other hand, if $\phi(t)=0$ on a subinterval of time, then we would have a singular control.\\
\\
For optimal control problems with state constraints there is the possibility of boundary singular sub-arcs as well as interior singular sub-arcs. Boundary singular sub-arcs occur when the constraint is tight, meaning when $E_G(t)=0$. If $E_G(t)=0$, that implies $E_G'=0$, and solving (\ref{myEG}) to obtain $f_{s,b}$,
%\begin{equation} 
%f_{s,b}=\frac{1000c_3j(N)}{V(glyc(V))}\label{fsb}\ \ ,
%\end{equation}
provided that $V(glyc(V)) \neq 0$. Note that as $glyc(V)$ accounts for the percent of glycogen energy being used and there is always a percentage of fuel usage coming from glycogen, this function will never be $0$. Due to the structure of the differential equation for $V$, as soon as $f$ is positive, $V$ is also positive. Hence, $V(glyc(V)) \neq 0$ when the boundary singular arc is active. 
\\
\\
If the constraint is not tight, then we have the possibility of an interior singular sub-arc that we can obtain by differentiating the switching function with respect to time twice, noting that on an interior singular sub-arc $\eta(t)=0$, and solving for $f_{singular}$.\\
When the constraint $E_G(t)$ is not tight, $\eta(t)=0$; however, when it is tight we solve for it using the fact that over the small time interval where the glycogen energy constraint is tight, the switching function, $\phi(t)$, is $0$ on a boundary subarc and thus also $\phi'(t)=0$. We are then able to solve this equation for $\eta(t)$ to get our complete characterization of $\eta(t)$. \\
\noindent For singular subarcs, one incurs an additional necessary condition known as the generalized Legendre-Clebsch (GLC) condition \cite{GLC}. The necessary GLC condition for the existence of a singular subarc for an optimal control problem, in which the control appears in $\phi''(t)$ is
$$\frac{\partial}{\partial f}\frac{d^2}{dt^2}\frac{\partial}{\partial f}(\mathcal{L})\ = \ \frac{\partial}{\partial f}\frac{d^2\phi}{dt^2}\ \ \ \geq \ \ 0,$$\\
a condition one can verify for this problem.
For our last necessary condition we have that the Transversality Conditions:\ \ $\lambda_1(T)= \lambda_2(T)= \lambda_4(T)=0\ .$
%Since we require $E_G=0$ at the final time we do not have a transversality condition for $\lambda_3$ . Instead, $\lambda_3$ is a constant over the entire time interval as  $E_G >0$ until the end and $\eta=0$ until the end.\\  
%\\
From the above, our optimal control, force, has the following structure:\\

\begin{align} 
f(t) = \left\{ \begin{array}{ccc}
               0 & for & \phi(t) < 0\nonumber\\
                \ & \ &\nonumber\\
                f_{singular}(t)  & for & \phi(t)=0\,\ E_G \neq 0 \nonumber\\
                    \ & \ &\label{foc}\\
               f_{s,b}(t)  & for &  \phi(t)=0\,\ E_G = 0 \nonumber\\
                    \ & \ & \nonumber\\
              f_{max} & for & \phi(t) > 0\ \nonumber \\
                \end{array}\nonumber \right.
\end{align}
\\ 
This $f(t)$ represent the force profile of which a runner should adhere to in order to run the optimal race. We have shown that our problem meets all of the necessary conditions for existence of a singular interior sub-arc, including the GLC condition, which suggests that the optimal trajectory could include a singular component.  To summarize the control trajectory, we know that the optimal control begins with a maximum force sub-arc, that for a race with large enough $T$, an optimal control must be comprised of more than just a maximum force arc, that a singular boundary force sub- arc exists (and likely at the end of the time interval), that a singular interior sub-arc is likely, and that intuitivly, it is unlikely for a zero force arc to exist, as that would not be optimal. 
%\\
%Something we should be able to answer is: what is the optimal control for our model if there is no nutritional input?  Our question then reduces to that of others that does not include nutritional input during the race; however, we do have different energy equations.\\
We believe that for our system, the control is singular for the majority of the event, and is comprised of a maximum force sub-arc, followed by a singular interior sub-arc, and finishing with a singular boundary sub-arc, as has been shown to be the trajectory in simpler systems in \cite{Keller1,Pitch,Aft,Wood}.  Unfortunately one is unable to easily obtain an exact solution structure with switching times. As the problem is linear in the control, and believed to be largely singular, along with having state costraints and the nutrition pulse function, it is challenging to solve even using adjoints and penalty multipliers. Thus, we chose to discretize the system in order to implement obtaining an approximate solution to the optimal control problem.  In the next section we discuss our nutrition strategies. We then outline our approximation method where we discretize the problem and add a penalization technique to reduce the noise in the solution. 
\subsection{Nutrition Strategies}
Not only are we concerned with what force $f$ value will maximize the distance over a fixed time T, but also with which in race nutrition strategy, $s_i$ is optimal.  As there are many philosophies within the running community about how often you should take nutrition during a race, we test 15 different nutrition strategies, $S=\{s_0,s_1,s_2,\dots, s_{15}\},$ shown in Table \ref{geltable}. \\
%The general form of the corresponding source function is: 
%
%\begin{align*}
%s_i(t)&=(418b)rectpuls(t-(w+h(0)),p)+(418b)rectpuls(t-(w+h(1)),p)...\\
%&+(418b)rectpuls(t-(w+h(n)),p)\ ,
%\end{align*}
%
%where $h(j)$ denotes the timing of the gels the runner is taking, w represents the first time the gel is taken, $418b$ represents how much energy is in each gel, b represents how many gels are taken at a given time,  and $p$ represents how long it takes for the gel to be consumed.\\
\\
During Eulid Kipchoge's race to the beat the 2 hour marathon barrier, his nutrition consisted of hydrogels by Maurteen. He consumed 100g Carbs per hour (which would be 4 baked sweet potatoes in solid food for reference) \cite{sub2}, altough there is no comment on his exact feeding regime. This amount of carbohydrates is pretty high for a runner to tolerate during a marathon due to the increased gastrointestinal movement in running. This number of carbohydrates is more in line with what cyclist would consume during their longer races. It is expected that at his level, his nutrition provided him a 1-3 percent boost in performance compared to him solely consuming water during the race. This is in contrast to the practice most runners follow of taking about 50 carbs per hour, an area of research that is continuously being studied by nutrition companies. We have many different nutrition strategies over which we are optimizing, two of which include simulating Kipchoge's known race intake, and a standard intake strategy used by average runners.  We will look at the percentage improvement from using the determined strategy versus consuming no carbohydrates.  As a reminder, we do not have stomach sensitivity to food throughout the race in our model yet, which could skew the results in favor of taking as much nutrition in as possible.\\
\begin{table}[]
\caption{Scenarios of nutrition intake during a marathon race}
\begin{center}
\begin{tabular}{ |c|p{8cm}| } 
 \hline
 $s_0$,$s_1$,$s_2$,$s_3$, $s_4$,$s_5$,$s_6$, $s_7$ & 100 Calorie gels spread evenly 0,1,2,3,4,5, 11, and 24  times throughout the race.  \\
 \hline
 $s_8$  & One 100 Calorie gel taken towards the beginning of the race.  \\
\hline
 $s_9$  & One 100 Calorie gel taken towards the end of the race. \\
\hline
 $s_{10}$  &Four 100 Calorie gels: 2 taken early in the race, 2 taken toward the end of the race. \\
\hline
 $s_{11}$  & Four 200 Calorie gels spread evenly throughout the race. \\
\hline
 $s_{12}$, $s_{13}$ & 250 Calorie gels spread evenly 2,4 times throughout the race  \\
\hline
 $s_{14}$  & Ten 50 Calorie gels taken evenly throughout the race.  \\
\hline
 $s_{15}$  & One 250 Calorie gel taken, 1 100 Calorie gel taken, 1 250 Calorie gel taken.  \\
\hline
\end{tabular}
\end{center}
\label{geltable}
\end{table}
\section{Methods}
Over time, optimal control problems have been solved and approximated using many different numerical techniques. There are many numerical examples in mathematical biology that use the Forward Backwards sweep method \cite{hackbusch, Lenhart}. Programs such as GPOPS and PASA have been developed to handle particular types of optimal control problems \cite{HagerZ,Rao}. %When the switching structure of the optimal control is known, many have cast optimal control problems as multiple boundary value problems. 
MATLAB  has a minimization tool, fmincon, that is built to handle a variety of optimization problems. We tried to use the forward backwards sweep method, the packages GPOPS and PASA, as well as fmincon in the continuous setting; however, none of these methods were robust enough to handle this particular problem. Thus, we discretized our problem and used fmincon, inputting our differential equation system in through equality constraints.
We began the discretization of our optimal control problem by partitioning our time interval, $[0,T]$ using M+1 equally spaces nodes, $0=t_0<t_1<\cdots <t_M=T$. We used a left rectangular approximation for the objective functional, obtaining the maximization problem: 
\begin{equation}\label{discJ}
\max_f J(f)=\max_f[\sum_{k=0}^{M-1}hV_k]
\end{equation}
where $h=\frac{T}{M}$\ , with bounded controls:\ \ $0\leq f_k\leq f_{max}=36000$ meters/min$^2$,\ and $f=(f_0,\dots , f_{M-1})$, with $s=(s_1,\dots,s_{15})\in S$, and $0\leq s_{i}(t_k)\leq s_{max}$ for all $0<k<M-1$\ , and with initial conditions:\ \ $V_0=0$, $E_{F,0}=K$, $E_{G,0}=144$, $N_0=0$.\\
\\
Next we use a forward Euler approximation for the state equations and obtain:\\
\begin{align}
V_{k+1}&=V_k+h(f_k-\frac{V_k}{\tau})\label{discV}\\
&\nonumber\\
E_{F,k+1}&=E_{F,k}+h(-af_kV_k(1-glyc(V_k)))\label{discEF}\\
&\nonumber\\
E_{G,k+1}&=E_{G,k}+h(c_3j(N_k)-af_kV_k(glyc(V_k)))\label{discEG}\\
&\nonumber\\
N_{k+1}&=N_k+h(s(t_k)-dN_k-j(N_k))\label{discN}
\end{align}
We also now have the discretized version of $J(N_k)$:\ \ $j(N_k)=c_4N_k\ .$
% \begin{align*}
% s(t_k)&=418rectpuls(t_k-(w+h(0)),p)+418rectpuls(t_k-(w+h(1)),p)+\cdots\\
% &+418rectpuls(t_k-(w+h(n)),p)\ ,\\
% \end{align*}
% which is equal to some $s_i$ depending on the values of $w,h,p$. 
We discretized our glyc function, $glyc(V_k)$,  by using a spline interpolator in MATLAB, described in the next section, and dropped in each $i^{th}$ nutrition scenario, $s_{k}$, corresponding to the time points each nutrition strategy designated. To optimize our system, while satisfying the constraints, we use fmincon as the minimization solver on the discretized system, using a left-rectangular integral approximation for the objective function.\\% Although we went through the process of discretizing the entire system, including the adjoints, to use Matlab's minimization solver fmincon, we do not supply it with adjoints.  
%\newpage
%\begin{figure}[h!]
%\begin{center}
%    \includegraphics[scale=.7]{fminconadjustVelo2} \\
%    \vspace{-3.75cm}
%\end{center}
%\caption{Approximation of the Optimal Control Problem with adjusted velocity equation to penalize the runners velocity when $E_G < 1$, with fmincon and control force, $f$, a vector.}
%\label{glue}
%\end{figure}
%\newpage
\\
``Fmincon is a gradient-based method that is designed to work on problems where the objective and constraint functions are both continuous and have continuous first derivatives" \cite{Matlab}, but due to our constraints and nutrition input, it was helpful to write it as a discrete system.  Fmincon is set to accept: the objective function, a starting guess for the control vector, X0, inequality linear constraints, equality linear constraints, lower and upper bounds for the state variables, nonlinear inequality and equality constraints, as well as options that allow the user to change the MATLAB settings for various features or provide the system with more information. 
% This set up differed from our original attempt at using fmincon in that here, we put each of the discretized state equations in the optimization as an equality constraint.
Our discretized objective function was originally $$J=-\delta\sum_{i=1}^{T-1}V_k$$ \\where $V_k$ occurs in our $X$ vector as entries $X(3n+1:4n)$ and $\delta=\frac{T}{M}$ where $T$ is the length of the race and $M$ is the number of mesh points, including $t=0$, but we will modify this objective function to include minimizing the total variation in $f$.  In order to modify the objective function to include minimizing the total variation, we added 2 additional state variables, $\iota$ and $\zeta$ \cite{Atkins}, such that the variation in 2 time points in $f$ is written as $$f(i+1)-f(i)=\zeta(i)-\iota(i)\ .$$
\\
Our starting guess vector, $X0$, that gives an initial placeholder for our state vector was: $X0= [\overbrace{ 36000\ \dots 36000}^T\ \overbrace{0\ \dots 0}^{6T-2}].$ Note that each state vector variable is a column vector of length $(T \times 1)$ and the two vectors used for variation penalization, $\zeta$ and $\iota$ are of length $(T-1) \times 1$ due to their structure. Our initial condition vector was:\ \ $[f(1)\ E_F(1)\ E_G(1)\ V(1)\ N(1)\ \zeta(1)\ \iota(1)]=[36000\ 3439\ 144\ 0\ 0\ 0\ 0]\ .$ As we do not have any inequality constraints and our equality linear constraints were written as:\ \ $A_{eq}x=b_{eq}$, where $Aeq$ is a $n \times n$ matrix, $x$ is our solution column vector of length $n \times 1$, and $beq$, of length $n \times 1$ is the righthand side of our equality linear constraints.
%We formulate $A_{eq}$ and $b_{eq}$ and all other components that make up this matrix and vector in the appendix. \\
\\
\\
As the velocity and nutrition equations are linear, we input those as our linear equality constraints. Our 2 energy equations are non-linear and thus are input as non-linear equality constraints.  One important feature for our problem is our energy constraints, $E_G(i)\geq 0$ and $E_F(i)\geq0$ for every $i\in[0,T]$. With our problem set up in its current format, we are able to simply input our lower bounds for our fat and glycogen energies as 0. We input the runners physical force, velocity, and energy capacities as appropriate upper bounds. We also give our variables appropriate initial conditions for their starting out fat and glycogen energies, 0 velocity, and 0 in race nutrition.\\ 
\\  
In our glycogen energy equation we have a function we call ``glyc" and its counter part ``1-glyc" in our fat energy equation.
% Our ``glyc" function is an approximation of what percent fuel utilization the runner is using from their glycogen supply dependent on what percentage of VO2max the runner is traveling at, and therefore ``1-glyc" approximates what percent of fuel utilization is coming from the runners fat supply.
We approximated these from a fuel utilization Figure \ref{3glyc} adapted from one in Stipanuk and Caudill's textbook  \cite{biochem}, by creating 2 vectors of points where one vector represents the percent VO2 max and the other represent the percent glycogen used.
%An example of our points is $$P=[0\ 0.1\ 0.2\ 0.4\ 0.6\ 0.8\ 0.9\ 1.0\ 1.1\ 1.2\ 1.6\ 2]$$ $$Q=[0.1\ 0.2\ 0.3\ 0.48\ 0.65\ 0.80\ 0.92\ 0.95\ 0.99\ 0.99\ 0.99\ 0.99]$$
%The $P$ vector represents the percent VO2 max, and $Q$ vector represents percent glycogen/fat respectively.
To smooth out our glyc function, we used a function in MATLAB called spline.  Spline is a cubic piecewise polynomial interpolator that is continuous and twice differentiable everywhere. It takes the 2 vector of points you supply it with and creates $n-1$ cubic polynomials that connect at the supplied $x$ vector given.\\
% Supplying spline with a larger set of $x$ points allows for a closer fit to our assumed structure. With fewer $x$ points you get more drastic turns in the cubics. The spline that was created from the vectors mentioned above, can also be found in our model section.\\
%\\
%\\
%As will be evident in the preliminary results, one remaining issue was variation in the control. This is sometimes called chattering or noise. This phenomenon often occurs in problems that are linear in the control where the optimal control has a singular sub-arc. When this occurs one will see a lot of variation, or jumps, in the control variable over the course of the time interval.  We ran into this very issue in our problem.  These jumps could be noise, perhaps caused by this problem being largely singular throughout the middle of the time interval.  It is certainly unreasonable to assume a runner can be running their max velocity one second, immediately go to a much slower velocity for a second, before returning to some other high velocity.  Changing your propulsive force in such a way would be extremely tiring and infeasible.\\
\\
One common issue in control problems where the control has a singular sub-arc is variation. 
%There have been penalization methods created and used to handle this situation. 
Ding and Lenhart \cite{LenhartD} as well as Caponigro et al. \cite{Cap} both penalized their original objective functional to regularize the chattering.  In order to obtain a reasonable trajectory for our runner, we chose to penalize our objective functional by adding a term to our objective functional that bounded the total variation \cite{Atkins}. Adding this penalty, not only amends the objective function to include penalizing for variation, it adds a 5th constraint to our problem. By adding this penalty we are decreasing the total variation in the solution trajectory. We can express the total variation, $\mathcal{V}(f)$, in the control \cite{Atkins} as: $$ \mathcal{V}(f)=\sup \sum_{i=0}^{N-1}| f(i+1)-f(i)|$$ where the supremum is being taken over all possible partitions of our time interval. Due to the difficulty of differentiating absolute value functions we decompose the absolute value term, using 2 new $T-1$ vectors $\zeta(i)$ and $\iota(i)$, whose entries are non-negative. Each entry of $\zeta(i)$ and $\iota(i)$ will be defined as $$f(i+1)-f(i)=\zeta(i)-\iota(i)$$
 We can think of this decomposition as satisfying the following 2 conditions \cite{Atkins}:\\
 \\
  If $f(i+1)-f(i)>0$, then $\zeta(i)=f(i+1)-f(i)$ and $\iota(i)=0$;\\
  If $f(i+1)-f(i)\leq 0$, then $\zeta(i)=0$ and $\iota(i)=-(f(i+1)-f(i))$,\\
where $\zeta(i)\geq 0$, $\iota(i)\geq 0$ and either $\zeta(i)=0$ or $\iota(i)=0$
 \\
\\
Appending this penalty to the objective function we have:  $$J=-\delta\sum_{i=1}^{T-1}V(i)+ p\sum_{i=1}^{T-2}(\zeta(i)+\iota(i))\ ,$$ and resulted in the extra linear equality constraint $$ceq5(i)=  -f(i+1) + f(i) + \zeta(i) - \iota(i)\ .$$ 
\\
The coefficient $p$ of our bounded variation scales the degree to which we penalize variation.  If $p$ is large, there is a more emphasis being placed on minimizing the variation of the control variable, force.\\
%A code was written to check that the supplied gradients in the solver and the gradients obtained through the finite difference process in the solver only differed on the order of $10^{-6}$.   When running our marathon code that includes the nutrition compartment, with a number of different parameters, initial conditions, and for different amounts of time, the solver returned a minimum solution that satisfied the constraint, with a fmincon exit flag of 2. In the MATLAB documentation, one can see that when using fmincon on a system with any non-smooth pieces (like our nutrition input), the best exit flag one can hope for is a 2, where a minimum is found, but the solver stopped due to the step size being smaller than the tolerance \cite{Matlab}.  When we took nutrition completely out of the model the solver returned the minimum solution that satisfied the constraint, with an exit flag of 1 regardless of the variation coefficient value.  
\section{Results and Discussion}
\subsection{World record optimal results}
We found that using a direct discrete optimization without adjoint functions, fmincon, was the most efficient approximation method and captured all the dynamics. Table \ref{table:kipchoge} shows the constants and coefficients with their scientific meaning and their units. Five parameters that can be chosen depending on the individual are: $E_G(0)$ (Initial glycogen energy), mass, VLa type, nutrition uptake rate,$c_4$,and VVO2 max, shown at the end of the table.
%\\
%\begin{table}[ht!]
%	%\begin{center}
%	%\centering
%	\caption{Parameters for Marathon Runner Model}
%	\vspace{1cm}
%		\begin{tabular}{||c c c||}
%			\hline 
%			Parameter & Unit & Meaning \\ [0.5ex] 
%			\hline\hline
%			$T$ &minutes & length of race\\
%			\hline
%			M &unitless &  number mesh points including $t=0$ \\ % 0.0993
%			\hline
%			$\tau$&min$^{-1}$ & internal resistant force constant \\ %.1157
%			\hline
%			$d$ &min$^{-1}$ & loss of nutrition to non-muscular system \\ %0.4957
%			\hline
%			$c_3$&kg$^{-1}$& mass conversion constant \\ %65.1613
%			\hline
%			$p$& unitless & variation penalty coefficient\\
%			\hline
%			$a$&kilojoule/joule & unit conversion constant  \\ %.1354
%			\hline
%			$\delta$ &T/M& discretization parameter\\ %134.5685
%			\hline
%			$sm$  &(seconds)$^{-2}$ & seconds to minutes conversion \\ %65.3614
%			\hline
%			$E_G(0)$ & kilojoules/kilogram & initial energy at start of race\\
%			\hline
%			$m$& kilograms & runner mass  \\ %0.0716
%			\hline
%			$c_4$ &min$^{-1}$ & nutrition uptake rate \\ %0.4997
%  			\hline
%			$vvmax$ &meters/min & VVO2 max \\ %0.2286
%			\hline
%			VLa type  &unitless & good, average, bad varied in glyc function \\ %.8853
%			\hline
%			%	$\text{a}$, $\text{b}$, $\text{c}$, $\text{d}$, $\text{e}$, $\text{f}$ & --- & constant  \\ 
%			%	\hline
%		\end{tabular}
%		
%		\label{table:params}
%%	\end{center}
%\end{table}
%\newpage
We first chose to simulate the optimal race of the current world record holder.  In the attempt to break the 2 hour marathon barrier, experts collaborated to optimize a set of runners VLa max, VO2 max, general running economy, nutrition plan, as well as pick the perfect race course for these runners.  The parameter values for our simulation can be found in Table \ref{table:kipchoge}. Note that the 3 different VLa types correspond to runners with 3 different VLamax's. These particular parameter values were picked to reflect appropriate values for a professional (world record breaking) runner. Plots of the state and control variables are seen in Figure \ref{kipfig}. \\
\\
In Figure \ref{kipfig} we see that the runner's velocity rapidly increases from 0 to reach an approximately constant velocity of $357$ m/min, that they can maintain for the entirety of the race.  The plot of the runner's propulsion force, similar to the runner's velocity plot, shows the runner's force quickly increasing from 0 m/min$^2$ to a constant force of $2.14 \cdot 10^4$ m/min$^2$. The runner's  fat energy, $E_F$, decreases linearly, from $3439$ KJ/Kg to $3390$, nowhere near the energy constraint, $E_F\geq 0$.  The runner's glycogen energy decreases from its initial value of $E_G(0)=144$ KJ/Kg to its final value of approximately $0$ KJ/Kg.  We see bumps in this subplot corresponding to the nutrition fueling.  In this simulation we assumed that the runner took four 200 calorie gels at $t=20$, $46$, $71$, $97$, which can be seen in the nutrition subplot in Figure \ref{kipfig}, and the runner took 120 minutes to complete the distance.\\ 
%\newpage
\begin{table}[]
	%\begin{center}
	\centering
	\caption{Parameters Used for World Record Holder}
	%\vspace{1cm}
		\begin{tabular}{||c c c c||}
			\hline 
			Parameter & Value & Unit & Meaning \\ [0.5ex] 
			\hline\hline
			$T$ & $120$&minutes & length of race\\
			\hline
			$M$ &$T+1$ & minutes &  number mesh points including $t=0$ \\ % 0.0993
			\hline
			$\tau$& $1/60$ &min & internal resistant force constant\\ %.1157
			\hline
			$d$ & $0.005$&min$^{-1}$ & loss of nutrition \\ %0.4957
			\hline
			$sm$& $1/3600$&(seconds)$^{-2}$ &seconds to minutes conversion\\ %0.0716
			\hline
			$c_3$& $1/m$&kg$^{-1}$& mass conversion constant \\ %65.1613
			\hline
			$p$&$0.5$& unitless & variation penalty coefficient \\
			\hline
			$a$& $1/1000$&kilojoules/joules & unit conversion constant  \\ %.1354
			\hline
			$\delta$ & $120/121$&T/M& discretization parameter\\ %134.5685
			\hline
			$m$  &$55$&kilogram & runner mass \\ %65.3614
			\hline
			$E_G(0)$& $144$ & kilojoules/kg & initial energy at start of race\\
			\hline
			$c_4$& 1/6 &min$^{-1}$ & nutrition uptake rate \\ %0.4997
  			\hline
			VVO2max& 402 &meters/min & Velocity at 100\% VO2max \\ %0.2286
			\hline
			VLa type& good  &unitless & good, average, bad  in glyc function \\ %.8853
			\hline
			%	$\text{a}$, $\text{b}$, $\text{c}$, $\text{d}$, $\text{e}$, $\text{f}$ & --- & constant  \\ 
			%	\hline
		\end{tabular}
		
		\label{table:kipchoge}
\end{table}
%\vspace{2.6cm}
%
%\begin{table}[ht!]
%
%\begin{center}
%	\caption{Total Distance Achieved by the World Record Holder}
%%	\vspace{1cm}
%		\begin{tabular}{||c c||}
%			\hline 
%			Race Time (min)&  Distance (km) \\ [0.5ex] 
%			\hline\hline
%			$120$&42.5\\
%			\hline
%
%
%		\end{tabular}
%		
%		\label{table:kipchogeR}
%	\end{center}
%\end{table}

\begin{figure}[htbp]
\vspace{-3.5cm}
\begin{center}
    \includegraphics[scale=.4]{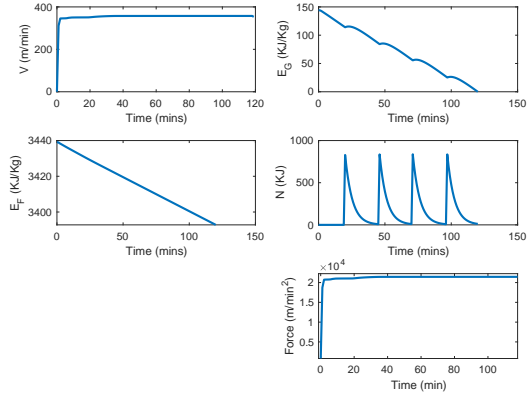} \\
    \vspace{-3cm}
\end{center}

\caption{World record marathon simulation, with four 200 calorie gels}
\label{kipfig}
\end{figure}
%\newpage
\noindent During the world record breaking race, the runner did not consume gels in the same manner that we implement, but did consume an equivalent amount of energy. The VVO2max was set to be $402$ meters per minute.  It is expected that world class runners have a VVO2 max in this range. This VVO2max is approximately equivalent to a runner running $4$min/mile.  We see that the distance achieved in our results is slightly longer than that of a marathon. The fastest marathon ever run thus far is 1:59:40. In the simulation from our model, the runner runs 120 minutes in 42.5 km, which is equivalent to the runner running a 1:59:09 marathon. This is only a $0.4$ percent difference between the current world record and our results.  This difference in average pace would be the runner running 4:33 minutes per mile instead of 4:34 minutes per mile. Note, that although our results are extremely close to the current world record, parameters could be changed to address a runner having an even better VLa, or better running economy. We also didn't explicitly have ``shoe type" in our model but a better shoe could result in better running economy which would result in a better VVO2 max.  It is believed that the world record time will continue to get faster over time.
\subsection{Results: Varying runner dependent parameters} 
 A runner's in-race nutrition, VLamax, and their VVO2max are the driving factors behind a runners performance, and thus it is important to see how changing these impacts the results. We first describe the results we obtained from simulations of our system with each of the 15 different nutrition strategies. The parameters used for these simulations are all the same ones used in Table \ref{table:kipchoge}, except, we changed $T=135$, which also changes $\delta=\frac{135}{136}$.\\
\\
Recall our 15 different nutrition strategies, as well as taking no nutrition labeled $s_0$, from Table \ref{geltable}. Table \ref{table:popp} shows the distance completed with each of these strategies. There are several conclusions we can draw from these results, which we will present after looking at graphs of several of the nutrition strategies. First we compare a strategies of taken 0 nutrition versus five 100 cal gels. From Table \ref{table:popp} we can see that the runner is able to run between $0.5$ and $0.7$ kilometers further per each additional gel, which translates to approximately $1.5\%$ increase in distance per gel. This means that by taking five 100 calorie gels instead of taking no nutrition, the runner has a $7.75\%$ improvement in performance. In Figure \ref{nogelsppt5} and Figure \ref{5gelsppt5} we see similar strategies in the force and velocity subplots; however the proplsion force and therefore the velocity that the runner who took 5 gels is able to maintain is significantly higher than that of the runner who took no nutrition.\\ 
\begin{figure}[htbp]
\vspace{-2cm}
\begin{center}
    \includegraphics[scale=.4]{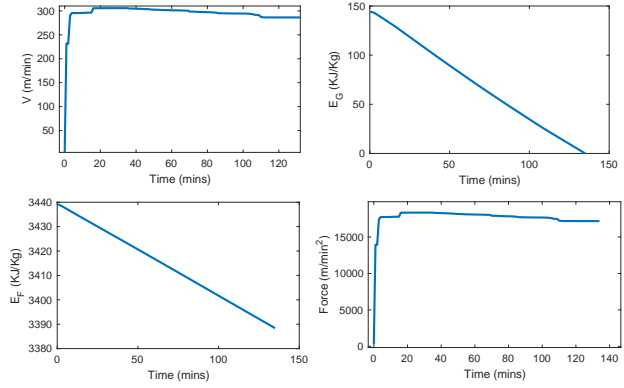} \\
    \vspace{-3cm}
\end{center}

\caption{Marathon simulation with 3 states and optimal force with no gels }
\label{nogelsppt5}
\end{figure}
\begin{figure}[htbp]
\vspace{-3cm}
\begin{center}
    \includegraphics[scale=.45]{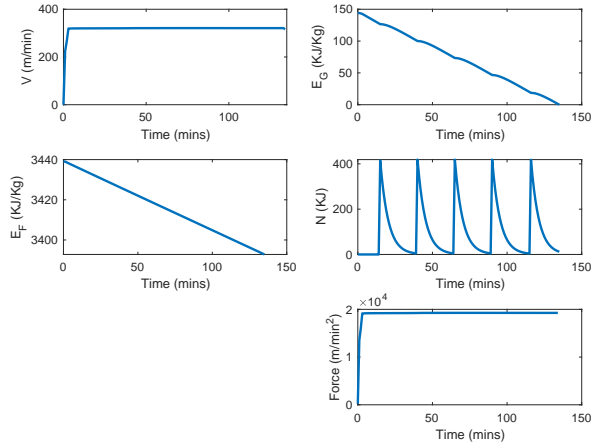} \\
    \vspace{-3cm}
\end{center}

\caption{Marathon simulation with 4 states and optimal force with five 100 calorie gels}
\label{5gelsppt5}
\end{figure}
\begin{table}[]
	\begin{center}
	%\centering
	\caption{Total Distance Achieved with Different Nutrition Strategies }
%	\vspace{1cm}
		\begin{tabular}{||c c c c||}
			\hline 
			Strategy  & Distance (km) & Strategy & Distance (km) \\ [0.5ex] 
			\hline\hline
			$s_0$& 40.0 & $s_8$&  40.7  \\ %0.0716
			\hline
			$s_1$ & 40.7 & $s_9$ & 40.5 \\
			\hline
			$s_2$ &  41.4 & $s_{10}$ & 42.5  \\ % 0.0993  \\
			\hline
			$s_3$ & 42.0 & $s_{11}$ & 45.1 \\ %.1157
			\hline
			$s_4$ &  42.6 & $s_{12}$ & 43.2 \\ %0.0716
			\hline
			$s_5$ &  43.1 & $s_{13}$ & 45.6 \\
			\hline
			$s_6$ &  46.0 & $s_{14}$ & 43.1 \\ % 0.0993
			\hline
			$s_7$ & 52.9 & $s_{15}$ & 43.7\\ %.1157
			\hline

		\end{tabular}
		
		\label{table:popp}
	\end{center}
\end{table}
%\newpage
%\begin{figure}[ht!]
%\vspace{-3cm}
%\begin{center}
%    \includegraphics[scale=.6]{hag54nogelsppt5} \\
%    \vspace{-4cm}
%\end{center}

%\caption{Marathon simulation with 3 states and optimal force with no gels}
%\label{nogelsppt5}
%\end{figure}
%
%\newpage
%
%\begin{figure}[ht!]
%\begin{center}
%    \includegraphics[scale=.65]{hag545gelsppt5} \\
%    \vspace{-4cm}
%\end{center}
%\caption{Marathon simulation with 4 states and optimal force with five 100 calorie gels}
%\label{5gelsppt5}
%\end{figure}

%\newpage
%\begin{figure}[ht!]
%\begin{center}
%%\vspace{-4cm}
%    \includegraphics[scale=.65]{hag544gelsppt5} \\
%\end{center}
%\vspace{-4cm}
%\caption{Marathon simulation with 4 states and optimal force, with four 100 calorie gels}
%\label{regn}
%\end{figure}
%
%\newpage
%\begin{figure}[ht!]
%\begin{center}
%    \includegraphics[scale=.65]{hag544200calgelsppt5} \\
%\end{center}
%\vspace{-4cm}
%\caption{Marathon simulation with 4 states and optimal force, with four 200 calorie gels}
%\label{pron}
%\end{figure}
%\newpage
\noindent We also checked if there would be any differences in performance if the runner took in the same amount of calories, but distributed differently. In $s_5$, where the runner took in 500 calories by taking five 100 calorie gels the distance achieved was 43.1 km. In $s_{12}$ where the runner takes in two 250 calorie gels the distance achieved was 43.2 km. Lastly, in $s_{14}$ where the runner consumes ten 50 calorie gels, the runner traveled 43.1 km. So, regardless of the way in which the runner consumes the same amount of calories, they travel relatively the same distance. Perhaps this would be different if there were a mechanism in the model that discouraged eating too much at once or eating later in the race (we attempt to remedy this in an extension of our model). Next, we compare strategy $s_5$, where the runner consumes a total of 500 calories, and strategy $s_8$, where the runner consumes a total of 800 calories.\\
\\
A runner taking five 100 calorie gels throughout the race is around the average number of gels used by your typical marathon runner. Some of the professionals are taking up to 100 carbohydrates an hour, which translates to 400 calories an hour. A runner finishing the marathon in about 2 hours would take in 800 calories, which is simulated with strategy $s_8$, in Figure \ref{kipfig}.  In the simulation the runner is able to run $1.5$km further, an improvement of $3.5\%$. In both simulations we see similar trajectories; however, the runner in simulation $s_8$ is able to run at a higher average speed throughout the race, yielding the larger distance ran in the same amount of time.  These results were expected as increasing the energy the body is able to use, should result in greater distance traveled.  This shows the importance in runners being able to take in as many calories during a race as possible.  It is possible to increase the amount of calories one can take in during a race.  One can do this by taking nutrition that is easier to digest, taking water with the nutrition so that the salinity balance of the gut is kept, and practicing taking in nutrition during long training runs.\\  %Next, we compare strategies, $s_1$,$s_8$, and $s_9$, where the runner takes one gel at different times in the race.
\\
Finally we take a look at the strategies, $s_{6}$, $s_{7}$ where the runner takes eleven and twenty-four 100 calorie gels throughout the race respectively. It is important to mention that these scenarios are unlikely to occur as a runner would struggle to consume this many gels without getting sick.  We completed simulations with nutrition strategies $s_{6}$ and $s_{7}$ to see if the model would organically capture the effect of a runner taking in too much nutrition. As can be seen in the results from Table \ref{table:popp}, this effect was not captured and the simulations simply show an increase in energy as more gels are added, and no negative side effects.  In the simulation where the runners consumes 11 gels, they are able to run 46.0 km, while in 24 gel simulation the runner travels 52.9 kms. Simply put, with a fixed time race, we see an increase in distance traveled for every gel that is taken. Taking in eleven and certainly twenty-four 100 calorie gels throughout a race would most certainly upset the runners stomach and all of those carbohydrates would struggle to make it to the muscles due to imbalances in the stomach.  The body does not handle the amount of accumulated energy in the nutrition compartment without a negative reaction. In future work we will handle this issue.\\

\noindent Next, we chose to vary the VLa types. We use the same parameters from Table \ref{table:kipchoge}, and vary the glyc function. 
%We also show what we project the optimal race for Eulid Kipchoge would be based on his known weight and rumored VO2max.\\
%\newpage
%\begin{table}[ht!]
%	%\begin{center}
%	\centering
%	\caption{Parameters Used for Different VLa Runs with Nutrition Included}
%%	\vspace{1cm}
%		\begin{tabular}{||c c c c||}
%			\hline 
%			Parameter & Value & Unit & Meaning \\ [0.5ex] 
%			\hline\hline
%			$T$ & $135$&minutes & length of race\\
%			\hline
%			$M$ &$T+1$ & minutes &  no. mesh points including $t=0$ \\ % 0.0993
%			\hline
%			$\tau$& $1/60$ &min & internal resistant force constant \\ %.1157
%			\hline
%			$d$ & $0.005$&min$^{-1}$ & loss of $N$ to non-muscular system \\ %0.4957
%			\hline
%			$c_3$& $1/m$&kilogram$^{-1}$& mass conversion constant \\ %65.1613
%			\hline
%			$p$&$0.5$& unitless & variation penalty coefficient \\
%			\hline
%			$a$& $1/1000$&kilojoule/joule & unit conversion constant  \\ %.1354
%			\hline
%			$\delta$ & $135/136$&T/M& discretization parameter\\ %134.5685
%			\hline
%			$sm$  &$1/3600$&(seconds)$^{-2}$ & seconds to minutes conversion \\ %65.3614
%			\hline
%			$E_G(0)$& $144$ &kilojoules/kilogram & initial energy at start of race\\
%			\hline
%			$m$& $55$&kilograms & runner mass  \\ %0.0716
%			\hline
%			$c_4$& $1/6$ &min$^{-1}$ & nutrition uptake rate \\ %0.4997
%  			\hline
%			$vvmax$& 402 &meters/min & VVO2 max \\ %0.2286
%			\hline
%			VLa type& varying  &unitless & good, average, bad glyc\\ %.8853
%			\hline
%			%	$\text{a}$, $\text{b}$, $\text{c}$, $\text{d}$, $\text{e}$, $\text{f}$ & --- & constant  \\ 
%			%	\hline
%		\end{tabular}
%		
%		\label{table:paramVLa}
%\end{table}
\begin{table}[]
%\vspace{.5cm}
	\begin{center}
	%\centering
	\caption{Total Distance Achieved by Runners with Different VLa Types across Two Nutrient Strategies}
%	\vspace{1cm}
		\begin{tabular}{||c c c||}
			\hline 
			VLa type& Nutrition &Distance (km) \\ [0.5ex] 
			\hline\hline
			good&0 calories taken& 41.7\\
			\hline
			average & 0 calories taken& 40.0 \\
			\hline
			bad &0 calories taken & 37.8\\ % 0.0993
			\hline
			good& Four $100$ calorie gels & 44.1\\
			\hline
			average & Four $100$ calorie gels& 42.6 \\
			\hline
			bad & Four $100$ calorie gels& 40.5\\ % 0.0993
			\hline
		\end{tabular}
		
		\label{table:VLaR}
	\end{center}
\end{table}
%\newpage
%\vspace{-2cm}
Table \ref{table:VLaR} shows  the distances achieved using different VLa types across two different nutrition strategies. The VLa types labeled as ``good", ``average", and ``bad" correspond to the glyc functions shown in Figure \ref{3glyc}.
%In Figures \ref{vgood} and \ref{vbad} are resulting graphs of the states and optimal control using the different VLa types.
For these runs, we tested the distance for the case where the runner takes in no nutrition during the race, as well as the runner taking in four $100$ calorie gels.
%\\
%%\newpage 
%\begin{figure}[h!]
%\begin{center}
%    \includegraphics[scale=.65]{hag54goodvla4gelsppt5} \\
%\end{center}
%\vspace{-4cm}
%\caption{Marathon simulation for runner with a good VLa, with four 100 calorie gels}
%\label{vgood}
%\end{figure}
%\newpage
%\begin{figure}[h!]
%\begin{center}
%    \includegraphics[scale=.65]{hag54badvla4gelsppt5} \\
%\end{center}
%\vspace{-4cm}
%\caption{Marathon simulation for runner with a bad VLa, with four 100 calorie gels}
% \label{vbad}
%\end{figure}
In the three different simulations with varying VLa's, the runner consumes four 100 calorie gels during the race.  In these results, the importance of runners training their body's to have lower maximum rate production of lactate for a long distance runner is evident.  The runner with the good VLa type ( i.e. the runner whose $glyc$ function was made to use fewer carbohydrates than fats than the runners with average or bad VLa's) is able to run much further in the same amount of time. There is an $8.9\%$-$10.3\%$ improvement in the runner with a good VLa as compared with a bad VLa depending on if the runners took in nutrition as well.  We see the same race structure regardless of the VLa type, with both simulations showing the runners ending with approximately $0$ glycogen energy, but similarly to the simulation where the runner takes more nutrition, the runner with the higher VLa is able to maintain a higher force.  The runner with a good VLa can maintain a force of $1.97\times10^4$ m/min$^2$ throughout the race, while the runner with a bad VLa was only able to maintain a level force of approximately $1.81\times10^4$ m/min$^2$. Some of these results can be seen in Table \ref{table:paramdifflevelsR}\\
\\
Lastly, we consider varying the runner's VV02max. As this term applies to all levels of runners, we vary many of the parameters including: the length of race, $T$, the runners mass, $m$, the initial glycogen energy, $E_G(0)$, the runners velocity at $100\%$ VO2 max, VVO2max, and their VLa type.
We use many of the same parameters from Table \ref{table:kipchoge}, but change the weight of the runner, $m$, length of race, $T$, as well as the VVO2max of the runners. Note that $c_3$ and $M$ also vary, but do so due to their relationship with $m$ and $T$ respectively.  Runners of different levels will finish the marathon in different lengths of time, thus the need to vary $T$. Also, runners of different abilities can have significantly different VVO2max values and VLa types, and therefore change depending on the individual.  Table \ref{table:paramdifflevelsR} shows the distances achieved for our different levels of runners with 4 lengths of time. We completed these simulations for marathon runners finishing in different lengths of time using fmincon optimization.\\
\\ 
%The resulting graphs of the states and control can be seen in Figures \ref{p} and \ref{b}.
For runners who are running for 155 minutes and 180 minutes with VVO2max's of 320 and 250 respectively, we see realistic results.  The runners run at around $85\%$ of their VO2 max that is suggested by experts. We continue to see the impact of taking in-race nutrition compared without nutrition.  The runner who takes 180 minutes to finish the race and has an average VLa has a $4\%$ improvement when they take nutrition versus when they do not.  One issue with these results is with the ``slowest runner", since they are running at a pretty high percentage of their VO2 max.  It is unlikely that a runner would be able to do that for this length of race.  This is more likely due to the fact that our model does not penalize a runner for running too fast, as long as they have enough glycogen energy. Recall, that our only constraints are that the energies have to be above $0$, and that there are some upper bounds on forces and velocities.  We will amend our current model to work for shorter distance races to address this issue of runners only being constrained by total amount of glycogen energy.  

\begin{table}[]
	%\begin{center}
	\centering
	\caption{Total Distance Achieved with Different Levels of Runners and Four Lengths of Time }
%	\vspace{1cm}
		\begin{tabular}{||p{0.9cm} p{1.5cm} p{1cm} p{0.9cm} p{2.6cm} p{1.5cm} p{0.9cm}||}
			\hline 
			Race Time (min)& VVO2 max (m/min)& weight (Kg)& VLa type& Nutrition (Cals) & $E_G(0)$ (KJ/Kg)& Dist. (Km) \\ [0.5ex] 
			\hline\hline
			$155$&320& 73&avg&4 100cal gels&150&42.9\\
			\hline
			$155$ &320& 73&avg& 0 100cal gels&150&40.1 \\
			\hline
			$155$ &  320&73&good&4 100cal gels&150&44.4\\ % 0.0993
			\hline
			$155$ & 320&73&bad & 4 100cal gels&150&41.6\\ %.1157
			\hline
			$180$ &  250&80&avg&0 100cal gels&140&41.0\\ %0.0716
			\hline
			$180$ &  250&80&avg& 4 100cal gels& 140&42.7\\
			\hline
			$215$ &  200&80&bad&0 100cal gels&140&41.6\\ % 0.0993
			\hline
			$210$& 200&80&avg&4 100cal gels&144 & 44.3\\
			\hline

		\end{tabular}
		
		\label{table:paramdifflevelsR}
%	\end{center}
\end{table}
%\newpage 
%\begin{figure}[ht!]
%\vspace{-2cm}
%\begin{center}
%    \includegraphics[scale=0.5]{hag54me4gelsppt5} \\
%\end{center}
%\vspace{-3cm}
%\caption{Marathon simulation for semi-professional (non-beginner) with four 100 calorie gels and race length=155min}
%\label{p}
%\end{figure}
%%\newpage
%\begin{figure}[hb!]
%\vspace{-2cm}
%\begin{center}
%    \includegraphics[scale=0.5]{hag54boston4gelsppt5} \\
%\end{center}
%\vspace{-3cm}
%\caption{Marathon simulation for a runner trying to qualify for the Boston Marathon: race length = 180mins and four 100 calorie gels.}
%\label{b}
%\end{figure}
%\newpage
\section{Conclusions and Future Work} 
%We have discussed the background literature, motivation behind our marathon model, our model, formulation of optimal control problems, optimization techniques, and some particular results.  We completely changed the way that energy is viewed in a runner model, by considering allocation for fat and carbohydrate energy dependent on the velocity of the runner.  Using fmincon on a discrete version of our problem we achieved similar propulsion force trajectories seen in previous works and we were also able to implement in-race nutrition into our model. We obtained really good results for the professional runner simulation compared to the current world record marathon of only $0.4\%$ error.\\
%\\    
%In the following chapter we share results from many simulations with optimization, varying different parameters as well as nutrition strategies.  We further discuss varying the bounded variation coefficient that was needed in our fmincon optimization.  We consider how our model can be adapted so that it is appropriate for distances other than the marathon, specifically shorter races like the 5km race. We also explore limitations of our marathon model, and add mechanisms to address those limitations.  At the end of chapter 3, we give conclusions of all the work from chapters 2 and 3. 

To understand the impact of nutrition we formulated our model such that in-race nutrition could be accounted for in the energy differential equations.  We created a marathon model described by velocity, fat energy, glycogen energy, and in-race nutrition differential equations. This marathon model better represented the body's energy systems and included nutrition pulses throughout the race.  After formulating an optimal control problem with this system of differential equations, and with the goal to maximize the distance achieved over a fixed time interval, we proved the existence of an optimal control. We determined a discretized approximate solution to the associated optimal control problem using MATLAB's fmincon optimization software, that included a bounded variation penalty on our control force.  We then optimized over a finite set of in-race nutrition strategies, obtaining a solution to our model for each strategy. \\
% After obtaining several results, varying different parameters and comparing some of our simulations to world records, we created two extended models to penalize a runner for running above their lactate threshold and to penalize a runner for overeating during a race.\\
\\
We analyzed the complicated role that in-race nutrition input has on a marathon runner as well as energy allocation dependent on the percentage of one's VVO2max during running.  Expressing the dynamics of the body's energy throughout a running race as 2 differential equations, representing fat energy and glycogen energy, with usage determined by an allocation function dependent on the ratio of one's velocity compared to their VVO2max, is a valuable contribution to the field of runner models.  Prior to our model, all energy dynamics during running races had been described based on available oxygen per unit mass does not allow for one to include in-race nutrition energy.  As a runner has a much smaller storage supply of glycogen than fat, and the body prefers to use the glycogen at high levels of exertion, the in-race nutrition consists of essentially only carbohydrates.  These factors required us to split up the energy into the two compartments and to use the allocation function.   While we obtained an optimal solution for an expert runner, we also individualized our model such that it can be used to determine a race strategy for any type of runner as long as they know their weight, VVO2max and VLamax type.\\
\\
To obtain our approximate solution to the optimal control problem of maximizing distance for a fixed time, we atte discretized our problem, and then optimized, using Matlab's optimizing tool, fmincon.  
%For the fmincon setup, we input our system of equations as equality constraints, and were able to include the upper and lower bounds on our state variables. 
We included a penalization for variation in the objective function.  The variation penalization was extremely important and the solution structure was very sensitive to this penalization.  With too small of a variation penalty coefficient, the runners purposive force was often erratic and non-optimal. We found that the optimal trajectory for marathon runners was to run at a steady state velocity for the majority of the race and take in at least four 100 calorie supplements spread evenly throughout the event.   Over a variety of runner's individual parameters, this racing strategy was optimal.  A runner can improve their performance with consistent and targeted training to improve their VV02max and their VLamax, (which correlates to having a better $glyc$ function) which can in turn drastically improve the their performance potential. When we approximated a solution to our problem for a runner with the parameters of the current world record holder in the marathon, our approximation was quite good with an error of $0.4\%$.\\   
\\
 Our contributions to the field are: formulating a system that includes energy allocation dependent on the ratio of their velocity and their velocity at VO2max, adding in-race nutrition to a runner model, and  approximating a solution to the optimal control problem, amended to include a penalty bounding the variation of the control. As we have found that taking more carbohydrates throughout the race improves performance, it is natural to believe that one should take in as many carbohydrates as possible; however, this is unrealistic as the stomach's digestion slows through a running race and also struggles to tolerate high volume of carbohydrates in the form of glucose and fructose.  Our initial marathon model does not account for this fact and our results showed that a runner's performance increased linearly with each gel they consumed.  To account for this limitation requires additional work, which will be presented in our next paper. \\
\\
Upon executing different scenarios for our marathon model, we realized that our model was not appropriate for races between a sprint and a half marathon.  In these scenarios we obtained results that were much faster than has been accomplished by humans thus far.  These results came from only having energy constraints in model, when in reality there is a difference in energy in the body and energy that is available for use in the body, dependent on lactate production.  Our extended model with the effects of lactate will be presented in our followup paper. In the future we would like to incorporate hydration into our model. Hydration levels can be a huge factor for a runner.  Low hydration can negatively affect the rate at which the body uptakes the nutrition. One possibility could be changing our nutrition uptake function to represent low, medium, or high levels of hydration. Building this feature in would be the next step in advising runners on optimal intake during running races.% Once we incorporated this phenomenon into our model we obtained results that were much closer to current world records.\\

\section{Bibliography}

\section*{Acknowledgements} 
%An Acknowledgements section is started with \verb"\ack" or
%\verb"\acks" for \textit{Acknowledgement} or
%\textit{Acknowledgements}, respectively. It must be placed just
%before the References.
We would like to acknowledge the help from Greg Hillson, Bobby Holcombe, and Summer Atkins.

%\bibliographystyle{siamplain}
%\bibliography{references}

\end{document}